\documentclass{article}

\usepackage{amsmath, amssymb, amsthm}
\usepackage{mathrsfs}   
\usepackage[all]{xy}    
\usepackage{graphicx}   
\usepackage{enumerate}

\usepackage{geometry}
\usepackage{layout}
\numberwithin{equation}{section} 

\newtheorem*{them*}{Theorem}
\newtheorem*{themA}{Theorem A}
\newtheorem*{themB}{Theorem B}
\newtheorem*{themC}{Theorem C}
\newtheorem*{themD}{Theorem D}
\newtheorem*{themE}{Theorem E}
\newtheorem*{mtA}{Main Theorem A}
\newtheorem*{mtB}{Main Theorem B}
\newtheorem{lemm}{Lemma}[section]
\newtheorem{defi}[lemm]{Definition}
\newtheorem{them}[lemm]{Theorem}
\newtheorem{prop}[lemm]{Proposition}
\newtheorem{cor}[lemm]{Corollary}
\newtheorem{rem}{Remark}

\begin{document}

    \title{On the higher algebraic $K$-groups of arithmetically equivalent number fields}

    \author{Ryo Komiya\footnote{Department of Mathematics, Institute of Science Tokyo, 2-12-1, O-okayama, Meguro-ku, Tokyo 152-8551, Japan, 
e-mail: \texttt{komiya.r.ab@m.titech.ac.jp}}}

    \date{} 

    \maketitle

    \begin{abstract}
        In this paper, based on the structure of higher algebraic $K$-groups of the rings of integers of number fields, we introduce new equivalence relations between number fields called $K$-equivalence and almost $K$-equivalence, and investigate their relationships with arithmetical equivalence and local integral equivalence. Building upon the classical properties of arithmetically equivalent number fields studied by R. Perlis and on the pioneering work by Komatsu, we apply the Rost-Voevodsky theorem (the Quillen-Lichtenbaum conjecture) for odd primes $p$, thereby analyzing algebraic $K$-groups within the framework of continuous \'{e}tale cohomology from a more modern perspective.
        
        As our main results, by utilizing permutation representations of global Galois groups and the data of local decomposition groups, we refine Komatsu's previous results in the range $p \neq 2$ and describe the conditions for number fields to be (almost) $K$-equivalent. Through these developments, we clarify how the arithmetic information reflected in higher $K$-groups resonates with the special values of zeta functions and Galois representations.
    \end{abstract}

    \section{Introduction}
    
    Let $F$ be a number field, and let $\mathcal{O}_F$ denote the ring of integers of $F$. Quillen’s algebraic $K$-groups ${K_n(\mathcal{O}_F)}$ generalize classical arithmetic invariants such as the ideal class group and the unit group to higher degrees, and their structure encodes deep arithmetic information about the number field. In this paper, for two number fields $k_1$ and $k_2$, if there is an isomorphism $K_n(\mathcal{O}_{k_1}) \simeq K_n(\mathcal{O}_{k_2})$ for all non-negative integers $n$, we call them \textit{$K$-equivalent}; on the other hand, if there is an isomorphism $K_n(\mathcal{O}_{k_1})[p^\infty] \simeq K_n(\mathcal{O}_{k_2})[p^\infty]$ for almost all primes $p$ independent of $n$, we call them \textit{almost $K$-equivalent}, where, for an abelian group $G$, the notation $G[p^\infty]$ denotes the $p$-primary part of $G$. One of the goals of this paper is to find a sufficient condition for (almost) $K$-equivalence. In the Main Theorem A, we will show that a sufficient condition for almost $K$-equivalence is "arithmetical equivalence". Furthermore, in Main Theorem B, we presented a sufficient condition for $K$-equivalence.

    \begin{mtA}[Corollary \ref{wengrtuan3oin}]
        If $k_1, k_2$ are number fields such that $\zeta_{k_1}(s) = \zeta_{k_2}(s)$, then all algebraic $K$-groups of odd degree are isomorphic:
  \[
    K_{2n+1}(\mathcal{O}_{k_1}) \simeq K_{2n+1}(\mathcal{O}_{k_2}).
  \]
  Furthermore, for any prime $p$ not dividing $2\nu_{k_1,k_2}$, there is an isomorphism of $K$-groups
  \[
    K_{2n}(\mathcal{O}_{k_1})[p^\infty] \simeq K_{2n}(\mathcal{O}_{k_2})[p^\infty].
  \]
  Consequently, $k_1$ and $k_2$ are almost $K$-equivalent. In particular, if either $k_1$ or $k_2$ is totally imaginary, the isomorphism $K_{2n}(\mathcal{O}_{k_1})[p^\infty] \simeq K_{2n}(\mathcal{O}_{k_2})[p^\infty]$ holds for any prime $p$ not dividing $\nu_{k_1,k_2}$.
    \end{mtA}

    The definition of the number $\nu_{k_1,k_2}$ is given  before Corollary \ref{wengrtuan3oin}.

    \begin{mtB}[Corollary \ref{nr4githto4wtio4at}]
        If $k_1,k_2$ are locally integrally equivalent (see the definition immediately preceding Corollary \ref{nr4githto4wtio4at}) and totally imaginary then $k_1,k_2$ are $K$-equivalent. 
    \end{mtB}

Two number fields are said to be \textit{arithmetically equivalent} if they share the same Dedekind zeta function. Historically, the equivalence of number fields has been extensively studied from the viewpoint of arithmetical equivalence. Moreover, refinements of the Neukirch–Uchida theorem are proved using properties of arithmetic equivalence, and this equivalence plays an important role in the context of anabelian geometry as well. Since the zeta functions of $k_1/\mathbb{Q}$ and $k_2/\mathbb{Q}$ cannot coincide when both are Galois extensions, we usually restricts attention to number fields that are not Galois.

R. Perlis\cite{Perlis} demonstrated that arithmetically equivalent number fields share invariants such as the degree, the discriminant, and the number of roots of unity, whereas he also clarified that an isomorphism between the ideal class groups themselves is not necessarily induced\cite{Perlis3}. Nevertheless, Perlis proved that a pair of arithmetically equivalent number fields possesses nearly isomorphic ideal class groups:

\begin{them*}[Perlis, 1978]
  Let $k_1, k_2$ be arithmetically equivalent number fields, and let $L$ be the common Galois closure of $k_1$ and $k_2$ over $\mathbb{Q}$. Then for any prime $p$ coprime to $\nu_{k_1,k_2}$, there is an isomorphism of abelian groups\cite[Theorem 3]{Perlis2}:
  \[
    \operatorname{Cl}(k_1)[p^\infty] \simeq \operatorname{Cl}(k_2)[p^\infty].
  \]
\end{them*}

If the number $\nu_{k_1,k_2}$ is defined as $\sharp \operatorname{Cl}(k_1)\sharp \operatorname{Cl}(k_2)$, this claim would be obvious; however, since $\nu_{k_1,k_2}$ in fact has the property of dividing the extension degree $[L:k_1]$, the claim is nontrivial ($L$ is the common Galois closure of $k_1,k_2$ over $\mathbb{Q}$).

It is well known that the algebraic $K$-groups of the ring of integers $\mathcal{O}_F$ recover the ideal class group at the torsion part of degree $n=0$, and the unit group at degree $n=1$:
\[
  K_0(\mathcal{O}_F) \simeq \mathbb{Z} \oplus \operatorname{Cl}(F), \qquad K_1(\mathcal{O}_F) \simeq \mathcal{O}_F^\times.
\]
In general, algebraic $K$-groups of even degree serve as generalizations of the ideal class group, while those of odd degree serve as a generalization of the unit group. Bearing this in mind, for a pair of arithmetically equivalent number fields, one expects an analogue of Perlis's theorem to hold for $K$-groups of even degree, and an isomorphism to hold for $K$-groups of odd degree. This implies that the equivalence relation of almost $K$-equivalence is induced by arithmetical equivalence. Indeed, Komatsu\cite{Komatsu} proved that arithmetically equivalent number fields are almost $K$-equivalent:

\begin{them*}[Komatsu, 1990]
  Let $k_1, k_2$ be arithmetically equivalent number fields, and let $L$ be the common Galois closure of $k_1$ and $k_2$ over $\mathbb{Q}$. For any prime $p$ not dividing $[L:k_1]$ and for any non-negative integer $n$, there is an isomorphism of abelian groups:
  \[
    K_n(\mathcal{O}_{k_1})[p^\infty] \simeq K_n(\mathcal{O}_{k_2})[p^\infty].
  \]
\end{them*}

The purpose of this paper is to refine Komatsu's result within the range $p \neq 2$, while elucidating how the condition that number fields are locally integrally equivalent or arithmetically equivalent imposes constraints on the structure of higher algebraic $K$-groups, particularly on their $p$-power torsion parts. Specifically, for an odd prime $p$, we adopt a method that describes algebraic $K$-groups in terms of continuous \'{e}tale cohomology, based on the Quillen-Lichtenbaum conjecture (which is now resolved by Rost and Voevodsky). Furthermore, based on the result of Corollary \ref{nr4githto4wtio4at} (Main Theorem B), we show that there exist infinitely many pairs of non-isomorphic and $K$-equivalent number fields (Corollary \ref{rhgtiwo94gniojwtnjowt}). Through these investigations, we highlight how the arithmetic information carried by higher $K$-groups resonates with the special values of zeta functions and permutation representations of Galois groups.

To this end, \S 2 organizes the theory of the complete group ring $\mathcal{O}[[G]]$ over a profinite group $G$ and coinduced modules, and reviews foundational theorems in Galois cohomology (such as Shapiro's lemma and the Mackey decomposition). In \S 3, we uses U. Jannsen's continuous \'{e}tale cohomology theory to precisely formulate the relationship between the \'{e}tale cohomology of the ring of integers and higher $K$-groups. In particular, we reduce the behavior of cohomology groups in extensions of number fields to a representation-theoretic comparison of global Galois groups and local decomposition groups. In \S 4, we will explain the main results and its corollaries.

    \vspace{0.2cm}

    \textbf{Notation and Conventions.} For a number field $F$, let $\mathcal{O}_F$ denote the ring of integers of $F$. For a ring $R$, let $K_n(R)$ denote Quillen's $n$-th algebraic $K$-group. For an abelian group $G$, let $G[p^\infty]$ denote the $p$-primary part of $G$, and let $G_{\rm tors}$ denote the torsion subgroup of $G$. For a finite group $G$, let $\sharp G$ denote the order of $G$.

    \vspace{0.2cm}
    \textbf{Acknowledgments.} The author would like to express their gratitude to Yuichiro Taguchi, Ryoji Shimizu and Tadashi Ochiai for their valuable feedback and insightful advice on the first draft of this paper.
    
    \section{Preliminaries}

Let $\mathcal{O}$ be the ring of integers of a $p$-adic field, and let $G$ be a profinite group. The ring
\[
\mathcal{O}[[G]] := \varprojlim_U \mathcal{O}[G/U],
\]
where the projective limit runs over the open normal subgroups $U$ of $G$, is called the completed group algebra. Let $\mathscr{C}_G$ denote the category whose objects are topological $\mathcal{O}[[G]]$-modules equipped with a compact Hausdorff topology, and whose morphisms are continuous $\mathcal{O}[[G]]$-module homomorphisms. In this section, we review classical facts concerning coinduced modules and reformulate them in the language of the category $\mathscr{C}_G$.

For a closed subgroup $H$ of a profinite group $G$, define the coinduction functor $\operatorname{Coind}_H^G : \mathscr{C}_H \to \mathscr{C}_G$ by
\begin{align*}
  \operatorname{Coind}_H^G(M)
  := \operatorname{Hom}_{\mathcal{O}[[H]]}^{\rm cts}(\mathcal{O}[[G]],M)
  \simeq \operatorname{Map}_H^{\rm cts}(G,M). \quad (M \in \operatorname{Ob}(\mathscr{C}_H))
\end{align*}
Here, for $f \in \operatorname{Hom}_{\mathcal{O}[[H]]}^{\rm cts}(\mathcal{O}[[G]],M)$, the action of $g \in \mathcal{O}[[G]]$ is defined by $(g\cdot f)(x) := f(xg)
\; (x\in \mathcal{O}[[G]])$, and by equipping this module with the compact-open topology, we regard it as an object of $\mathscr{C}_G$. This functor is right adjoint to the restriction functor $\operatorname{Res}_H^G : \mathscr{C}_G \to \mathscr{C}_H$.

For two objects $M,N \in \mathscr{C}_G$, define their completed tensor product by
\[
  M \widehat{\otimes}_{\mathcal{O}} N
  := \varprojlim_{U,V} M/U \otimes_{\mathcal{O}} N/V,
\]
where the projective limit runs over open submodules $U$ of $M$ and $V$ of $N$. By endowing this module with the diagonal action of $G$, the module $M \widehat{\otimes}_{\mathcal{O}} N$ becomes an object of $\mathscr{C}_G$. (Note that an open submodule $U$ of $M$ is not necessarily $G$-stable, so the action of $G$ cannot in general be defined on $M/U$. Hence we take the projective limit only over $G$-stable open subgroups. Since the collection of such subgroups is cofinal among all open subgroups, the resulting action of $G$ is well defined.)

This tensor product is equipped with a morphism in $\mathscr{C}_G$
\[
  \alpha: M \times N \longrightarrow M \widehat{\otimes}_{\mathcal{O}} N
\]
satisfying the following universal property: for every compact $\mathcal{O}$-module $X$ and every $\mathcal{O}$-bilinear map $F: M \times N \to X$, there exists a unique continuous homomorphism of topological $\mathcal{O}$-modules $G: M \widehat{\otimes}_{\mathcal{O}} N \to X$ such that $F = G \circ \alpha$. For $m\in M$ and $n\in N$, we denote the element $\alpha(m,n)$ of $M \widehat{\otimes}_{\mathcal{O}} N$ by $m\otimes n$.

Let $H$ be a closed subgroup of $G$, and let $N$ be an object of $\mathscr{C}_G$ and $M$ an object of $\mathscr{C}_H$. Define the map $\theta$ by
\[
  \theta:
  \operatorname{Coind}_H^G(M)\widehat{\otimes}_{\mathcal{O}} N
  \longrightarrow
  \operatorname{Coind}_H^G
  (M\widehat{\otimes}_{\mathcal{O}}\operatorname{Res}_H^G N);
  \qquad
  \theta(f\otimes n)(g)
  :=
  f(g)\otimes (g\cdot n).
\]
Then $\theta$ is well defined as a morphism in $\mathscr{C}_G$. 

    \begin{lemm}[Projection formula]
        \label{gtjwi4tjmo3tji4ot}
      Let $N$ be an object of $\mathscr{C}_G$ and $M$ an object of $\mathscr{C}_H$. If $H$ is an open subgroup of $G$, then the morphism
    \[
      \theta: \operatorname{Coind}_H^G(M) \widehat{\otimes}_{\mathcal{O}} N \longrightarrow \operatorname{Coind}_H^G(M \widehat{\otimes}_{\mathcal{O}}\operatorname{Res}_H^G N); \quad \theta(f\otimes n)(g) := f(g) \otimes (g \cdot n)
    \]
      is an isomorphism.
    \end{lemm}

    \begin{proof}
      Since $H$ is open, there exists a finite set $T \subseteq G$ of representatives for the right cosets of $H$ in $G$ (with $|T|=:d$). Each coset $Ht$ is an open and closed subset of $G$, and hence restriction of functions induces an isomorphism
      \[
        \operatorname{Map}_H^{\rm cts}(G,M)
        \simeq
        \bigoplus_{t\in T}\operatorname{Map}_H^{\rm cts}(Ht,M).
      \]
      Under this isomorphism, we obtain
      \[
        \begin{array}{ccccc}
\operatorname{Coind}_H^G(M) = \operatorname{Map}_H^{\rm cts}(G,M)
& \stackrel{\sim}{\longrightarrow} &
\bigoplus_{t \in T} \operatorname{Map}_H^{\rm cts}(Ht,M)
& \stackrel{\sim}{\longrightarrow} &
\bigoplus_{t\in T}M\\
\qquad\qquad\qquad\rotatebox{90}{$\in$}
&
&
\rotatebox{90}{$\in$}
&
&
\rotatebox{90}{$\in$}\\
\qquad\qquad\qquad f
& \longmapsto &
(f|_{Ht})_t
&\longmapsto  &
(f(t))_t.
\end{array}
      \]
      By exactly the same argument, we also obtain an isomorphism
      \[
      \operatorname{Coind}_H^G(M \widehat{\otimes}_{\mathcal{O}}\operatorname{Res}_H^G N)
      \overset{\sim}{\longrightarrow}
      \bigoplus_{t\in T}
      M \widehat{\otimes}_{\mathcal{O}}\operatorname{Res}_H^G N.
      \]

      Now define a morphism in $\mathscr{C}_G$
      \[
        \Theta:
        \left(\bigoplus_{t \in T} M\right)
        \widehat{\otimes}_{\mathcal{O}} N
        \longrightarrow
        \bigoplus_{t\in T}
        \left(M \widehat{\otimes}_{\mathcal{O}}\operatorname{Res}_H^G N\right)
      \]
      by $(m_t)_t \otimes n
        \mapsto
        (m_t\otimes t \cdot n)_t$. This is clearly an isomorphism, whose inverse is given by
      \[
        (m_t\otimes n_t)_t
        \longmapsto
        \sum_{t\in T} m_t \otimes t^{-1}\cdot n_t.
      \]

      Therefore, the diagram
      \[
        \xymatrix{
          \operatorname{Coind}_H^G(M) \widehat{\otimes}_{\mathcal{O}} N \ar[r]^-{\theta} \ar[d]^-{\wr} &
          \operatorname{Coind}_H^G(M \widehat{\otimes}_{\mathcal{O}}\operatorname{Res}_H^G N) \ar[d]^-{\wr} \\
          (\bigoplus_{t \in T} M) \widehat{\otimes}_{\mathcal{O}} N \ar[r]^-{\Theta} &
          \bigoplus_{t\in T} (M \widehat{\otimes}_{\mathcal{O}}\operatorname{Res}_H^G N)
        }
      \]
      is commutative, and hence $\theta$ is an isomorphism.
    \end{proof}

    \begin{lemm}[Mackey decomposition]
        Let $H$ and $K$ be closed subgroups of a profinite group $G$, and let $M \in \mathscr{C}_H$ be an $H$-module. If either $H$ or $K$ is an open subgroup, then there exists an isomorphism in $\mathscr{C}_K$
        \[
          \operatorname{Res}_K^G\operatorname{Coind}_H^G M
          \simeq
          \bigoplus_{g \in H \backslash G / K}
          \operatorname{Coind}_{K\cap g^{-1}Hg}^K \left({}^{g^{-1}}M\right).
        \]
        Here, ${}^{g^{-1}}M$ denotes the module on which an element $x = g^{-1}hg \in K\cap g^{-1}Hg$ acts on $m\in M$ by $ x\cdot m := h\cdot m$.
    \end{lemm}

    \begin{proof}
        Fix a complete set of representatives $g \in H \backslash G / K$ for the double cosets of $G$. Then $\operatorname{Coind}_H^G M
        \simeq
        \operatorname{Map}_H^{\rm cts}(G,M)$ admits a direct sum decomposition as a $K$-module:
        \[
          \operatorname{Coind}_H^G M
          \simeq
          \bigoplus_{g \in H \backslash G / K} V_g,
          \qquad
          V_g
          :=
          \left\{
          f \in \operatorname{Map}_H^{\rm cts}(G,M)
          \ ; \
          \operatorname{supp}(f) \subseteq HgK
          \right\}.
        \]

        For each representative $g$, define a map
        \[
          \psi_g :
          V_g
          \to
          \operatorname{Map}_{K\cap g^{-1}Hg}^{\rm cts}(K,{}^{g^{-1}}M)
        \]
        by $\psi_g(f)(k) := f(gk)$. For $x = g^{-1}hg \in K\cap g^{-1}Hg$, we compute
        \begin{align*}
            \psi_g(f)(xk)
            =
            f(gxk)
            =
            f(gg^{-1}hgk)
            =
            f(hgk)
            =
            h\cdot f(gk)
            =
            x \cdot \psi_g(f)(k),
        \end{align*}
        and hence this map is well defined.

        Conversely, given $\phi \in \operatorname{Map}_{K\cap g^{-1}Hg}^{\rm cts}(K,{}^{g^{-1}}M)$, we obtain an inverse map by defining $f(hgk) := h\cdot \phi(k)$. Therefore $\psi_g$ is an isomorphism, and the lemma follows.
    \end{proof}

    By the proof, if we define the morphism
        \[
          \mathrm{Mackey}:
          \operatorname{Res}_K^G\operatorname{Coind}_H^G M
          \to
          \bigoplus_{g \in H \backslash G / K}
          \operatorname{Coind}_{K\cap g^{-1}Hg}^K
          \left({}^{g^{-1}}M\right)
        \]
        as the composite
        \[
          f \in \operatorname{Map}_H^{\rm cts}(G,M)
          \mapsto
          (f|_{HgK})_g
          \mapsto
          (\psi_g(f|_{HgK}))_g,
        \]
        then this becomes an isomorphism. This isomorphism depends only on the choice of the complete set of representatives.

        \begin{lemm}
            \label{rejtgiownmgtiom}
            Let $H$ be an open subgroup of a profinite group $G$, and let $N$ be a closed normal subgroup of $G$ such that $N \subseteq H$. Given a module $A \in \mathscr{C}_H$ on which $H$ acts trivially, there is an isomorphism of $G$-modules
            \[
              \operatorname{Coind}_H^G(A)
              \simeq
              \operatorname{Inf}^G_{G/N}\left(\operatorname{Coind}_{H/N}^{G/N}(A)\right).
            \]
            Here, for a $G/N$-module $M$, $\operatorname{Inf}_{G/N}^G M$ denotes $M$ regarded as a $G$-module via the quotient map $G \twoheadrightarrow G/N$.
        \end{lemm}

        \begin{proof}
            The coinduced module can be written as
            \[
              \operatorname{Coind}_H^G(A)
              \simeq
              \{f: G \to A \ \text{(cont.)} \ ; \ \forall g \in G, \forall h \in H, \ f(hg) = h\cdot f(g)\}.
            \]
            Since $H$ acts trivially on $A$, we have $f(hg)=f(g)$. Hence $\operatorname{Coind}_H^G(A) \simeq \operatorname{Map}(H\backslash G,A)$. The action of $G$ on an element $\varphi$ of $\operatorname{Map}(H\backslash G,A)$ is given by right translation: $(g_0\cdot \varphi)(Hg) = \varphi(Hgg_0)$. Since $N$ is a normal subgroup of $G$, for any $n\in N$ and $g\in G$ we have $gng^{-1}\in N$. Thus $gn = (gng^{-1})g$, and hence $Hgn = Hg$.
            Therefore $n\cdot \varphi=\varphi$, so $N$ acts trivially on $\operatorname{Map}(H\backslash G,A)$.

            Using the bijection of left cosets
            \[
              H \backslash G
              \simeq
              (H/N)\backslash (G/N);
              \quad
              Hg \mapsto (H/N)(gN),
            \]
            we obtain
            \[
              \operatorname{Map}(H\backslash G,A)
              \simeq
              \operatorname{Map}((H/N)\backslash (G/N),A).
            \]
            Moreover, this isomorphism preserves the action by right translation. Hence the lemma follows.
        \end{proof}

Let $\mathscr{D}_G$ denote the category whose objects are topological $\mathcal{O}[[G]]$-modules equipped with a discrete Hausdorff topology, and whose morphisms are continuous $\mathcal{O}[[G]]$-module homomorphisms. This category is an abelian category with enough injectives, and hence one obtains the right derived functors of the left exact functor
\[
  \mathscr{D}_G^{\mathbb{N}}
  \to
  \mathbf{Ab};
  \quad
  (M_n,d_n)
  \mapsto
  \varprojlim H^0(G,M_n).
\]
Here, $\mathscr{D}_G^{\mathbb{N}}$ denotes the category of inverse systems in $\mathscr{D}_G$ indexed by the set $\mathbb{N}$ of natural numbers with the natural order. The fact that $\mathscr{D}_G$ has enough injectives is equivalent to the fact that $\mathscr{D}_G^{\mathbb{N}}$ has enough injectives\cite[(1.1) Proposition]{Jannsen}. We denote the $i$-th right derived functor by $H^i(G,(M_n,d_n))$. For an object $(T_n)_n$ of $\mathscr{D}_G^{\mathbb{N}}$ satisfying the Mittag--Leffler condition, there is a functorial isomorphism\cite[(2.2) Theorem]{Jannsen}:
\[
  H^i_{\rm cts}\left(G,\varprojlim_n T_n\right)
  \simeq
  H^i(G,(T_n,d_n))
\]
where, for a profinite group $G$ and a continuous $G$-module $A$, the notation $H^i_{\rm cts}(G,A)$ denotes the continuous group cohomology.

Suppose that a compact $G$-module $A$ is expressed as the projective limit of a projective system of finite discrete $G$-modules $\{A_i\}_{i=1}^\infty$. Then one naturally obtains a short exact sequence
        \[
          0
          \longrightarrow
          \varprojlim\nolimits_n^1 H^{i-1}(G,A_n)
          \longrightarrow
          H^i_{\rm cts}(G,A)
          \longrightarrow
          \varprojlim\nolimits_n H^i(G,A_n)
          \longrightarrow
          0
        \]
        \cite[Theorem~2.7.5]{NSW}. In general, the projective limit of cohomology groups $\varprojlim\nolimits_n H^i(G,A_n)$ is not necessarily isomorphic to the continuous cohomology $H^i_{\rm cts}(G,A)$, but when the Mittag--Leffler condition is satisfied, the term $\varprojlim\nolimits_n^1$ vanishes, and these cohomology groups become naturally isomorphic.

        \begin{lemm}
            If $H^{i-1}(G,A_n)$ is finite for every $n$, then
            \[
              H^i_{\rm cts}(G,A)
              \longrightarrow
              \varprojlim\nolimits_n H^i(G,A_n)
            \]
            is an isomorphism\cite[Corollary~2.7.6]{NSW}.
        \end{lemm}
Let $G$ be a profinite group, $H$ a closed subgroup of $G$, and $M$ an $H$-module. Consider the two homomorphisms
        \[
          \varphi: H \hookrightarrow G,
          \qquad
          \phi:\operatorname{Coind}_H^G(M) \to M;
          \quad
          f \mapsto f(1).
        \]
        These are compatible in the following sense: for any $f \in \operatorname{Coind}_H^G(M), \; h \in H$, one has
        \[
          \phi(\varphi(h)\cdot f)
          =
          \phi(h\cdot f)
          =
          (h\cdot f)(1)
          =
          f(h)
          =
          h\cdot f(1)
          =
          h\cdot \phi(f).
        \]
        Hence we obtain a homomorphism on cohomology groups
        \[
          \operatorname{sh}_1:
          H_{\rm cts}^n(G,\operatorname{Coind}_H^G(M))
          \to
          H_{\rm cts}^n(H,M).
        \]

        \begin{lemm}[Shapiro]
          \label{senti34wtioemto}
            The map $\operatorname{sh}_1$ is an isomorphism\cite[Proposition~(1.6.4)]{NSW}.
        \end{lemm}

        Now fix an element $g\in G$, and consider the following two homomorphisms:
        \[
          \varphi: g^{-1}Hg \hookrightarrow G,
          \qquad
          \phi:\operatorname{Coind}_H^G(M) \to {}^{g^{-1}}M;
          \quad
          f \mapsto f(g).
        \]
        Here, ${}^{g^{-1}}M$ denotes the $g^{-1}Hg$-module whose action is defined by $(g^{-1}hg)\cdot m := h\cdot m$
        for $m\in M$.

        As above, these maps are compatible: for any $h\in H$ and
        $
        f \in \operatorname{Coind}_H^G(M),
        $
        one has
        \[
          \phi(\varphi(g^{-1}hg)\cdot f)
          =
          \phi(g^{-1}hg\cdot f)
          =
          f(gg^{-1}hg)
          =
          f(hg)
          =
          h\cdot f(g)
          =
          (g^{-1}hg)\cdot \phi(f),
        \]
        where the last equality follows from the definition of the action on ${}^{g^{-1}}M$.
        Therefore we obtain a homomorphism on cohomology groups
        \[
          \operatorname{sh}_g:
          H_{\rm cts}^n(G,\operatorname{Coind}_H^G(M))
          \to
          H_{\rm cts}^n(g^{-1}Hg,{}^{g^{-1}}M).
        \]

        \begin{lemm}
            \label{rmgtiowmto3emtor}
            The diagram
        \[
          \xymatrix{
            &
            H_{\rm cts}^i(G,\operatorname{Coind}_H^G(M))
            \ar[ld]_-{\mathrm{sh}_g}
            \ar[rd]^-{\mathrm{sh}_1}
            & \\
            H_{\rm cts}^i(g^{-1}Hg,{}^{g^{-1}}M)
            &
            &
            H_{\rm cts}^i(H,M)
            \ar[ll]_-{\mathrm{conj}_g}
          }
        \]
        is commutative. Here, $\operatorname{conj}_g$ denotes the isomorphism on cohomology induced by $g^{-1}Hg \to H; \quad x \mapsto gxg^{-1}$,
        together with the identity map $M \to {}^{g^{-1}}M;
          \quad
          x \mapsto x$. Therefore, $\operatorname{conj}_g$ is also an isomorphism.
        \end{lemm}

          \begin{proof}
              If $A$ is a $G$-module, then the homomorphism $H^i(G,A)\to H^i(G,A)$ induced by
        \[
        G\to G;
        \:
        x\mapsto gxg^{-1},
        \quad
        \text{and}
        \quad
        A\to A;
        \;
        x\mapsto g^{-1}x,
        \]
        coincides with the identity map. Applying this to $A=\operatorname{Coind}_H^G(M)$, and composing before $\operatorname{sh}_g$, we see that $\operatorname{sh}_g$ is induced by the pair
        \[
          g^{-1}Hg \to G;
          \;
          x\mapsto gxg^{-1},
          \quad 
          \text{and}
          \quad 
          \operatorname{Coind}_H^G(M)
          \to
          {}^{g^{-1}}M;
          \;
          f \mapsto g^{-1}f \mapsto f(1).
        \]
        This clearly coincides with $\operatorname{conj}_g\circ \operatorname{sh}_1$.
          \end{proof}

        \section{The structure of algebraic $K$-groups}

In this section, we study the structure of the algebraic $K$-groups $K_{2n}(\mathcal{O}_F)$ associated with the ring of integers $\mathcal{O}_F$ of a number field $F$ in order to prove Theorem \ref{hguhtuoiwntoir3}, which asserts that if $k_1$ and $k_2$ are arithmetically equivalent number fields, then for every odd prime number $p$ relatively prime to the integer $\nu_{k_1,k_2}$ introduced by Perlis, there exists an isomorphism $K_{2n}(\mathcal{O}_{k_1})[p^\infty] \simeq K_{2n}(\mathcal{O}_{k_2})[p^\infty]$. The other corollaries mentioned in the introduction are obtained from this theorem through several additional arguments.

The basic idea of the proof of Theorem \ref{hguhtuoiwntoir3} is as follows. First, it is well known that there exists an isomorphism
\[
 K_{2n}(\mathcal{O}_F)[p^\infty]
\simeq
H^2_{\rm cts}(G_{F,S},\mathbb{Z}_p(n+1)) \quad (p \ne 2 \ \text{if} \ F \ \text{is not totally imaginary}),
\]
where $S$ denotes the set consisting of all primes of $F$ lying above $p$ together with all infinite places of $F$ (The notation is explained in the following paragraph). Assuming that $k_1$ and $k_2$ are arithmetically equivalent, we would like to derive an isomorphism between the two cohomology groups
\[
H^2_{\rm cts}(G_{k_1,S_1},\mathbb{Z}_p(n+1))
\quad \text{and} \quad
H^2_{\rm cts}(G_{k_2,S_2},\mathbb{Z}_p(n+1)),
\]
where $S_1$ and $S_2$ are defined similarly to $S$. However, it is difficult to compare these groups directly.

 To overcome this difficulty, we apply Shapiro’s lemma (Lemma \ref{senti34wtioemto}) to obtain cohomological isomorphisms of the form
\begin{equation}
\label{rngitowjgto4meioa4}
H^2_{\rm cts}(G_{k_i,S_i},\mathbb{Z}_p(n+1))
\overset{?}{\simeq}
H^2_{\rm cts}\bigl(G_{\mathbb{Q},T},
\operatorname{Coind}_{G_{k_i,S_i}}^{G_{\mathbb{Q},T}}
\mathbb{Z}_p(n+1)\bigr),
\end{equation}
where $T = \{p,\infty\}$ which is independent of $i$ ($i=1,2$). After this reduction, it remains only to compare the coefficient modules. When $k_1$ and $k_2$ are arithmetically equivalent, priveous work of Perlis imply that these coefficient modules are in fact isomorphic if $p$ is prime to $\nu_{k_1,k_2}$.

In practice, however, in order to derive an isomorphism \eqref{rngitowjgto4meioa4}, it is necessary that $G_{k_i,S_i}$ be an open subgroup of $G_{\mathbb{Q},T}$. However, if the extension $k_i/\mathbb{Q}$ has ramified primes outside $p$, this condition fails. To resolve this issue, we enlarge the set $S_i$ slightly and thereby bridge this gap.

  Let $F$ be a number field, and let $\mathcal{O}_F$ denote the ring of integers of $F$. We fix a separable closure $\mathbb{Q}^{\rm sep}$ of the rational number field $\mathbb{Q}$, and denote the Galois group $\operatorname{Gal}(\mathbb{Q}^{\rm sep}/\mathbb{Q})$ by $G_{\mathbb{Q}}$. We also fix an embedding $F \hookrightarrow \mathbb{Q}^{\rm sep}$ of $F$ into $\mathbb{Q}^{\rm sep}$ and set $G_F := \operatorname{Gal}(\mathbb{Q}^{\rm sep}/F)$. For each finite prime $v$ of the number field $F$, we fix a prime $\widetilde{v}$ in $\mathbb{Q}^{\rm sep}$ lying over $v$. We denote the decomposition group associated with this $\widetilde{v}$ by $D_{v} \subseteq G_F$. 

Let $P = P_F$ denote the set of all primes of the number field $F$, and let $P_\infty = P_{F,\infty}$ denote the set of all infinite primes of $F$. Also, let $P_{\rm f} = P_{F,\mathrm{f}}$ denote the set of all finite primes of $F$. For a rational prime $q$, let $P_q = P_{F,q}$ denote the set of all primes of $F$ lying over $q$. Let $F'$ be an extension field of $F$ such that $\mathbb{Q} \subseteq F \subseteq F' \subseteq \mathbb{Q}^{\rm sep}$. For a set of primes $S \subseteq P_F$ of $F$, we denote by $S(F')$ the set of all primes of $F'$ lying over $S$. That is, $S(F') := \{\mathfrak{p} \in P_{F'} ; \mathfrak{p}|_{F} \in S \}$. Moreover, for a set of primes $T \subseteq P_{F'}$ of $F'$, we denote by $T_F$ the set of all primes of $F$ lying under the primes in $T$. That is, $T_F := \{\mathfrak{p} \in P_F; \exists \mathfrak{P} \in T \ \mathrm{s.t.} \ \mathfrak{P}|_{F} = \mathfrak{p}\}$. We shall denote by $\operatorname{ram}(F'/F)$ the set of all primes of $F'$ that ramify in $F'/F$.
  
Let $S \subseteq P_F$ be a set of primes of $F$, and let $F_S$ denote the maximal extension of $F$ inside $\mathbb{Q}^{\rm sep}$ unramified outside $S$. Since this is a Galois extension of $F$, we denote its Galois group by $G_{F,S} := \operatorname{Gal}(F_S/F)$. For a non-archimedean prime $v$ of $F$, the continuous homomorphism induced by restriction $D_{v} \to G_{F,S}$ induces a map on cohomology, which we write as $\mathscr{R}es_v: H_{\rm cts}^i(G_{F,S},-) \to H_{\rm cts}^i(D_v,-) $. Furthermore, we denote by $D_{v,S} \subseteq G_{F,S}$ the decomposition group of the prime obtained by restricting $\widetilde{v}$ to $F_S$. Using the natural surjection $\pi: G_F \to G_{F,S}$, we have $D_{v,S} = \pi(D_v)$.

\begin{prop}
  \label{eni4vmo34vo}
  Let $p$ be an odd prime, and let $S$ be a finite set of primes of $F$ satisfying $S \supseteq P_\infty \cup P_{p}$. Then, for $n > 0$, there is an isomorphism of abelian groups
  \[
    K_{2n}(\mathcal{O}_F)[p^\infty] \simeq \operatorname{ker}\left[H^2_{\rm cts}(G_{F,S}, \mathbb{Z}_p(n+1)) \xrightarrow{\bigoplus \mathscr{R}es_v} \bigoplus_{v \in S - (P_\infty \cup P_p)}H^2_{\rm cts}(D_{v}, \mathbb{Z}_p(n+1))\right].
  \]
  If $F$ is totally imaginary, the isomorphism also holds for $p=2$.
\end{prop}

  \begin{proof}

    When $S = P_\infty \cup P_p$, this asserts that $K_{2n}(\mathcal{O}_F)[p^\infty] \simeq H^2_{\rm cts}(G_{F,S}, \mathbb{Z}_p(n+1))$. In this case, the claim can be shown as follows: for a non-negative integer $i$ and a non-empty finite set of primes $T \subseteq P_F$ of $F$, the group $H^i(\operatorname{Spec}(\mathcal{O}_{F,T}), \mu_{p^m}^{\otimes n+1})$ is finite for each $m$\cite[Proposition 2.13]{Milne}. Furthermore, by \cite[Proposition 2.9]{Milne}, there is an isomorphism between \'{e}tale cohomology and Galois cohomology for $r \geq 1$:
\[
  H^r(G_{F,T}, \mu_{p^m}^{\otimes n+1}) \simeq H^r(\operatorname{Spec}(\mathcal{O}_{F,T}), \mu_{p^m}^{\otimes n +1}).
\]
Since this isomorphism originates from the Hochschild-Serre spectral sequence, the diagram
\[
  \xymatrix{
    H^r(G_{F,T}, \mu_{p^{m+1}}^{\otimes n+1}) \ar[r]^-{\sim} \ar[d] & H^r(\operatorname{Spec}(\mathcal{O}_{F,T}), \mu_{p^{m+1}}^{\otimes n +1}) \ar[d] \\
    H^r(G_{F,T}, \mu_{p^m}^{\otimes n+1}) \ar[r]^-{\sim} & H^r(\operatorname{Spec}(\mathcal{O}_{F,T}), \mu_{p^m}^{\otimes n +1})
  }
\]
commutes naturally with respect to the morphism of sheaves $\mu_{p^{m+1}}^{\otimes n+1} \to \mu_{p^m}^{\otimes n+1}$. By taking the projective limit, we obtain the isomorphism
\[
  \varprojlim H^r(G_{F,T}, \mu_{p^m}^{\otimes n+1}) \simeq H^r(\operatorname{Spec}(\mathcal{O}_{F,T}), \mathbb{Z}_p(n+1)).
\]
The natural surjection $H^r_{\rm cts}(G_{F,T}, \mathbb{Z}_p(n+1)) \twoheadrightarrow \varprojlim H^r(G_{F,T}, \mathbb{Z}_p (n+1))$ becomes an isomorphism, because $H^{r-1}(G_{F,T} , \mu_{p^m}^{\otimes n+1})$ is finite for each $\mu_{p^m}^{\otimes n+1}$. From the above, we obtain the isomorphism of abelian groups
\[
  H^r_{\rm cts}(G_{F,T}, \mathbb{Z}_p(n+1)) \simeq H^r(\operatorname{Spec}(\mathcal{O}_{F,T}), \mathbb{Z}_p(n+1)).
\]
Applying this to $T = S$ and finally using the following theorem, the case $S = P_\infty \cup P_p$ is proved.

\begin{themA}
  For an odd prime $p$ and a positive integer $n$, there is an isomorphism
  \[
    K_{2n}(\mathcal{O}_{F})[p^\infty] \simeq H^2(\operatorname{Spec}(\mathcal{O}_{F}[1/p]), \mathbb{Z}_p(n+1)).
  \]
  \cite[Theorem 70]{Weibel3}. If $F$ is totally imaginary, the isomorphism also holds for $p=2$\cite[Theorem 73]{Weibel3}.
\end{themA}

\begin{rem}
  Regarding Theorem 70 and Theorem 73 in the literature\cite{Weibel3}, note that $K_{2n}(\mathcal{O}_F[1/p])[p^\infty] \simeq K_{2n}(\mathcal{O}_F)[p^\infty]$ holds for any number field $F$ and any rational prime $p$. From the localization exact sequence for algebraic $K$-groups, this yields a short exact sequence ($n > 0$)\cite[Theorem 7]{Weibel3}:
  \[
    0 \longrightarrow K_{2n}(\mathcal{O}_F) \longrightarrow K_{2n}(\mathcal{O}_F[1/p]) \longrightarrow \bigoplus_{\mathfrak{p}\mid p} K_{2n-1}(\mathcal{O}_F/\mathfrak{p}) \longrightarrow 0.
  \]
  Since $\mathcal{O}_F/\mathfrak{p}$ is a finite field, we can compute $K_{2n-1}(\mathcal{O}_F/\mathfrak{p}) \simeq \mathbb{Z}/((\sharp \mathcal{O}_F/\mathfrak{p})^n - 1)$\cite[Example 15]{Weibel3}. Therefore, taking the $p$-primary part of the short exact sequence yields the desired isomorphism.
\end{rem}

\begin{rem}
  Since $H^1(\operatorname{Spec}(\mathcal{O}_{F,T}), \mu_{p^m}^{\otimes n+1})$ is finite for each $m$ given a non-negative integer $i$ and a non-empty finite set of primes $T$ of $F$, the right-hand side of Theorem A may also be interpreted as the continuous \'{e}tale cohomology defined by Jannsen\cite{Jannsen}. 
\end{rem}

Therefore, in the following, we proceed with the proof of the theorem assuming that $S \supsetneq P_\infty \cup P_p$.

  Let $X := \operatorname{Spec}(\mathcal{O}_F[1/p])$, and define an open affine subscheme of $X$ by $U := \operatorname{Spec}(\mathcal{O}_{F,S})$. We also define a closed subscheme of $X$ by $Z := X - U$ (note that $Z$ is non-empty). For $n > 0$, we define \'{e}tale sheaves on $X$ by $\mathscr{F}_{m} := \mu_{p^{m}}^{\otimes n+1}$, which form an inverse system $\mathscr{F} = (\mathscr{F}_{m})_m$ with respect to $m$ (where $m$ runs over positive integers). We write the continuous \'{e}tale cohomology of Jannsen\cite{Jannsen} defined for this inverse system as $H^i_{\mathrm{cts}}(X,\mathscr{F}) = H^i_{\mathrm{cts}}(X,\left(\mathscr{F}_{m}\right))$. As the localization long exact sequence for continuous \'{e}tale cohomology, we obtain the exact sequence
\[
  \cdots \longrightarrow H^2_{Z,\mathrm{cts}}(X,\mathscr{F}) \longrightarrow H^2_{\mathrm{cts}}(X,\mathscr{F}) \longrightarrow H^2_{\mathrm{cts}}(U,\mathscr{F}) \overset{\delta}{\longrightarrow} H^3_{Z,\mathrm{cts}}(X,\mathscr{F}) \longrightarrow \cdots
\]
By the Gysin isomorphism\cite[Theorem 3.17]{Jannsen}, we have
\begin{align*}
  H^2_{Z,\mathrm{cts}}(X,\mathscr{F}) &\simeq \bigoplus_{v \in Z} H^2_{v,\mathrm{cts}}(X,\mathscr{F}) \\ 
  &\simeq \bigoplus_{v \in Z} H^0_{\mathrm{cts}}(\operatorname{Spec}(k(v)),\mathbb{Z}_p(n)) \\  
  &\simeq \bigoplus_{v \in Z} H^0(k(v),\mathbb{Z}_p(n)) \simeq \mathbb{Z}_p(n)^{G_{k(v)}} = 0,
\end{align*}
where $k(v)$ denotes the residue field at the closed point $v$. Bearing in mind that $K_{2n}(\mathcal{O}_F)[p^\infty] \simeq H^2_{\mathrm{cts}}(X,\mathscr{F})$, it follows that
\begin{align}
    \label{dfjhyg8i9w34tj04ejtmjiop}
  K_{2n}(\mathcal{O}_F)[p^\infty] \simeq H^2_{\mathrm{cts}}(X,\mathscr{F}) \simeq \operatorname{ker}\left( \delta: H^2_{\mathrm{cts}}(U,\mathscr{F}) \longrightarrow H^3_{Z,\mathrm{cts}}(X,\mathscr{F})\right).
\end{align}

For a closed point $v \in Z$, let $X_v := \operatorname{Spec}(\mathcal{O}_{F,v}^h)$ be the Henselization of $X$ at $v$. If we consider the pullback of the \'{e}tale sheaf $\mu_{p^m}^{\otimes n+1}$ on $X$ via the morphism of schemes $f: X_v \to X$, it coincides with the \'{e}tale sheaf $\mu_{p^m}^{\otimes n+1}$ on $X_v$. Therefore, by abuse of notation, we shall also denote the \'{e}tale sheaf $\mu_{p^m}^{\otimes n+1}$ on $X_v$ by $\mathscr{F}_{m}$. This yields a natural transformation of $\delta$-functors via the pullback:
\[
  f^* : H_{\mathrm{cts}}^i(X,\mathscr{F}) \longrightarrow H^i_{\mathrm{cts}}(X_v, \mathscr{F}).
\]

Now, we define $U_v := X_v \setminus \{v\} \simeq \operatorname{Spec}(F_v^h)$, which is obtained by removing the closed point $v$ from $X_v = \operatorname{Spec}(\mathcal{O}_{F,v}^h)$. By considering the pullback of sheaves under the morphism $f: X_v \to X$, we obtain the commutative diagram
\[
  \xymatrix{
    \cdots \ar[r] & H^2_{Z,\mathrm{cts}}(X,\mathscr{F}) \ar[r] \ar[d] & H^2_{\mathrm{cts}}(X,\mathscr{F}) \ar[r] \ar[d] & H^2_{\mathrm{cts}}(U,\mathscr{F}) \ar[r]^-{\delta} \ar[d] & H^{3}_{Z,\mathrm{cts}}(X,\mathscr{F}) \ar[r] \ar[d] & \cdots \\
    \cdots \ar[r] & H^2_{v,\mathrm{cts}}(X_v,\mathscr{F}) \ar[r] & H^2_{\mathrm{cts}}(X_v,\mathscr{F}) \ar[r] & H^2_{\mathrm{cts}}(U_v,\mathscr{F}) \ar[r]^-{\delta_{v}^{\rm loc}} & H^{3}_{v,\mathrm{cts}}(X_v,\mathscr{F}) \ar[r] & \cdots 
  }
\]
Here, we focus on the square
\begin{equation}
  \label{srejetnitmnot4t}
  \vcenter{
  \xymatrix{
     H^2_{\mathrm{cts}}(U,\mathscr{F}) \ar[r]^-{\delta} \ar[d] & H^3_{Z,\mathrm{cts}}(X,\mathscr{F}) \ar[d] \\
     H^2_{\mathrm{cts}}(U_v,\mathscr{F}) \ar[r]^-{\delta_v^{\rm loc}} & H^3_{v,\mathrm{cts}}(X_v,\mathscr{F})
  }
  }
\end{equation}
Since the diagram
\begin{equation}
  \label{serngtiow4trmio3mre}
  \vcenter{
  \xymatrix{
     H^3_{Z,\mathrm{cts}}(X,\mathscr{F}) \ar[d] \ar[r]^-{\sim} & \bigoplus_{v \in Z} H^3_{v,\mathrm{cts}}(X,\mathscr{F}) \ar[d]^{\mathrm{pr}_v} \\
     H^3_{v,\mathrm{cts}}(X_v,\mathscr{F}) & H^3_{v,\mathrm{cts}}(X,\mathscr{F}) \ar[l]^-{\sim}
  }
  }
\end{equation}
commutes (for the lower isomorphism, see \cite[(3.8) Proposition]{Jannsen}), we can define the map $\delta_v: H^2_{\mathrm{cts}}(U,\mathscr{F}) \to H^3_{v,\mathrm{cts}}(X_v,\mathscr{F})$ as the composition
\begin{equation*}
  H^2_{\mathrm{cts}}(U,\mathscr{F}) \xrightarrow{\delta} H^3_{Z,\mathrm{cts}}(X,\mathscr{F}) \xrightarrow{\sim} \bigoplus_{v \in Z} H^3_{v,\mathrm{cts}}(X,\mathscr{F}) \xrightarrow{\text{pr}_v} H^3_{v,\mathrm{cts}}(X,\mathscr{F}) \xrightarrow[\sim]{f^*} H^3_{v,\mathrm{cts}}(X_v,\mathscr{F}).
\end{equation*}
Consequently,
\begin{equation}
    \label{rnguiownmtoimt3e}
  \operatorname{ker}\delta = \bigcap_{v \in Z} \operatorname{ker}\delta_v
\end{equation}
holds. By combining diagrams \eqref{srejetnitmnot4t} and \eqref{serngtiow4trmio3mre}, and using \eqref{dfjhyg8i9w34tj04ejtmjiop} and \eqref{rnguiownmtoimt3e}, we obtain the following isomorphism:
\begin{equation}
    \label{erhjgio3wmjtoi4tjop}
  K_{2n}(\mathcal{O}_F)[p^\infty] \simeq \bigcap_{v\in Z}\operatorname{ker}\left[ H^2_{\rm cts}(U,\mathscr{F}) \to H^2_{\rm cts}(U_v,\mathscr{F}) \xrightarrow{\delta_v^{\rm loc}} H^2_{\rm cts}(X_v,\mathscr{F})\right]
\end{equation}

            \begin{lemm}
  The boundary map
  \[
    \delta_v^{\rm loc}: H^2_{\mathrm{cts}}(U_v,\mathscr{F}) \longrightarrow H^3_{v,\mathrm{cts}}(X_v,\mathscr{F})
  \]
  is an isomorphism.
\end{lemm}

\begin{proof}
  We cite the following theorem:

  \begin{themB}
    Let $S$ be the spectrum of a Henselian ring $A$, let $s_0$ be the closed point of $S$, let $\pi: X \to S$ be a proper morphism, and let $X_0 \to s_0$ be the closed fibre of $X/S$. If $F$ is a torsion sheaf on $X_{\rm \acute{e}t}$, then there is a canonical isomorphism
    \[
      H^i(X,F) \xrightarrow{\; \sim \;} H^i(X_0,F_0)
    \]
    where $F_0 = F|X_0$, for $i \geq 0$\cite[Corollary 2.7]{Milne}.
  \end{themB}

  Applying this theorem with $X = S = X_v$, we obtain the isomorphism
  \[
    H^i(X_v,\mathscr{F}_m) \simeq H^i(\operatorname{Spec}(k(v)),\mathscr{F}_m) \simeq H^i(k(v),\mu_{p^m}^{\otimes n+1})
  \]
  where $k(v)$ is the residue field of $v$. Since $k(v)$ is a finite field, its absolute Galois group is isomorphic to the additive group of the profinite completion of the integers $\widehat{\mathbb{Z}}$. In general, for a torsion $\widehat{\mathbb{Z}}$-module $M$, the cohomology groups can be computed as
  \[
    H^i(\widehat{\mathbb{Z}},M) \simeq \begin{cases}
        M^{\widehat{\mathbb{Z}}} & (i = 0) \\
        M_{\widehat{\mathbb{Z}}} & (i = 1) \\
        0 & (i \geq 2)
    \end{cases}
  \]
  \cite[Proposition 1.7.7]{NSW}. Therefore, $H^i(X_v,\mathscr{F}_{m})$ is finite for all $i$. Consequently, the $\varprojlim\nolimits_m^1$ term vanishes, and in general, we have an isomorphism for $r \geq 0$:
  \[
    H^r_{\rm cts}(X_v, \mathscr{F}) \simeq \varprojlim_m H^r(X_v,\mathscr{F}_{m}).
  \]
  It follows that $H^2_{\mathrm{cts}}(X_v,\mathscr{F}) = H^3_{\mathrm{cts}}(X_v,\mathscr{F}) = 0$, which implies that $\delta_v^{\rm loc}$ is an isomorphism.
\end{proof}

From this lemma and \eqref{erhjgio3wmjtoi4tjop}, we find that
\begin{equation}
    \label{dsrgtin4wntontioemrtoa}
  K_{2n}(\mathcal{O}_F)[p^\infty] \simeq \bigcap_{v \in Z}\operatorname{ker}\left[H^2_{\rm cts}(U,\mathscr{F}) \to H^2_{\rm cts}(U_v,\mathscr{F})\right]. 
\end{equation}

Let $\widetilde{U}$ be the universal covering of $U$, and let $\pi_{U}: \widetilde{U} \to U$ be the natural morphism. From the proof of \cite[Proposition 2.9]{Milne2}, it follows that $H^p(G_{F,S},(H^q(\widetilde{U},\pi_U^*\mathscr{F}_{m}))) = 0$ for $q > 0$. Therefore, the morphism
\[
  \Phi_U: H^2_{\rm cts}(\pi_1(U,z),\mathbb{Z}_p(n+1)) \longrightarrow H^2_{\mathrm{cts}}(U, \mathscr{F}) \quad (z := \operatorname{Spec}(\mathbb{Q}^{\rm sep}))
\]
obtained from the Hochschild-Serre spectral sequence in continuous \'{e}tale cohomology\cite[(3.3) Theorem]{Jannsen} is an isomorphism. Similarly, for $q > 0$, the \'{e}tale cohomology group $H^q(\widetilde{U}_v, \pi_{U_v}^*\mathscr{F}_{m})$ with respect to the universal covering $\widetilde{U}_v$ of $U_v$ is zero (where $\pi_{U_v}$ is defined analogously to $\pi_U$). Thus, the morphism
\[
  \Phi_{U_v}: H^2_{\rm cts}(\pi_1(U_v,z),\mathbb{Z}_p(n+1)) \longrightarrow H^2_{\mathrm{cts}}(U_v, \mathscr{F})
\]
obtained from the Hochschild-Serre spectral sequence is also an isomorphism.

       The restriction of $f: X_v \to X$ to $U_v$ yields a morphism $f_U: U_v \to U$, which induces a morphism on fundamental groups $f_{U*} : \pi_1(U_v,z) \to \pi_1(U,z)$. From the preceding discussion, we obtain a commutative diagram in which each row is an isomorphism:
\[
\xymatrix{
   H^2_{\rm cts}(\pi_1(U,z),\mathbb{Z}_p(n+1)) \ar[r]^-{\Phi_U}_-{\sim} \ar[d]^-{(f_{U*})^*} & H^2_{\mathrm{cts}}(U,\mathscr{F}) \ar[d]^-{f^*_U} \\
   H^2_{\rm cts}(\pi_1(U_v,z),\mathbb{Z}_p(n+1)) \ar[r]^-{\Phi_{U_v}}_-{\sim} & H^2_{\mathrm{cts}}(U_v,\mathscr{F}) 
  }
\]
Concerning the fundamental groups, the diagram
\[
  \xymatrix{
      D_v \ar[d]\ar[r]^-{\sim} & \pi_1(U_v,z) \ar[d] \\
      G_{F,S} \ar[r]_-{\sim} & \pi_1(U,z) 
  }
\]
commutes in a way that preserves the cyclotomic character (where the map $D_v \to G_{F,S}$ is the map induced by restriction). Consequently, we obtain the commutative diagram
\begin{equation}
\label{rdjgtiw0jtgot}
\vcenter{
\xymatrix{
   H^2_{\rm cts}(G_{F,S},\mathbb{Z}_p(n+1)) \ar[r]^-{\Phi_U}_-{\sim} \ar[d]^-{\mathscr{R}es_v} & H^2_{\mathrm{cts}}(U,\mathscr{F}) \ar[d]^-{f^*_U} \\
   H^2_{\rm cts}(D_v,\mathbb{Z}_p(n+1)) \ar[r]^-{\Phi_{U_v}}_-{\sim} & H^2_{\mathrm{cts}}(U_v,\mathscr{F}).
  }
}
\end{equation}
By the commutative diagram \eqref{rdjgtiw0jtgot} and the isomorphism \eqref{dsrgtin4wntontioemrtoa}, we can get the desired isomorphism
\[
  K_{2n}(\mathcal{O}_F)[p^\infty] \simeq \bigcap_{v \in Z} \left[H^2_{\rm cts}(G_{F,S},\mathbb{Z}_p(n+1) ) \xrightarrow{\mathscr{R}es_v} H^2_{\rm cts}(D_v,\mathbb{Z}_p(n+1)) \right].
\]
This completes the proof of the proposition \ref{eni4vmo34vo}.

\end{proof}

The proof of the following lemma was explained to the author by Ryoji Shimizu.

\begin{lemm}
    \label{drghiown4ogtmeo4mo}
  Let $v \in P_{F}$ be a finite prime of $F$ not lying above $p$. Then, the inflation map on cohomology
  \[
    H^2_{\rm cts}(D_{v,S},\mathbb{Z}_p(n+1)) \to H^2_{\rm cts}(D_v,\mathbb{Z}_p(n+1))
  \]
  is an isomorphism.
\end{lemm}

\begin{proof}
  Let $w$ denote the restriction to $F_S$ of the prime $\widetilde{v}$ of $\mathbb{Q}^{\rm sep}$ lying over $v$. Let $F_v$ be the localization of $F$ at $v$, and let $F_{S,w}$ be the localization of $F_S$ at $w$. We write $N_{v,S}$ for the kernel of the natural surjection $\pi: D_v \to D_{v,S}$.
  
  First, since $F(\mu_{p^\infty}) \subseteq F_S$, it is clear that $\mu_{p^\infty} \subseteq F_{S,w}$ holds. If we denote the Hilbert class field of $F$ by $H$, then $H \subseteq F_S$ because $H/F$ is unramified. By the principal ideal theorem, the ideal $v\mathcal{O}_H$ in $H$ generated by the prime $v$ becomes a principal ideal; thus, we choose a generator $\alpha \in \mathcal{O}_H$ for it. For a prime $w$ of $H$ lying over $v$, we have $\operatorname{ord}_w(\alpha) = 1$ since $v$ does not ramify in $H/F$. By Kummer theory, the field $H(\mu_{p^\infty},\alpha^{1/p^\infty})$ is unramified outside the primes dividing $\alpha$ (that is, the primes lying over $v$) and the primes dividing $p$. Because $v \in S$, it follows that $H(\alpha^{1/p^\infty}) \subseteq F_S$. Therefore, if we let $M$ denote the localization of $H$ at the prime obtained by restricting $w$ to $H$, we find that $M(\varpi^{1/p^\infty}_M) \subseteq F_{S,w}$ (where $\varpi_M$ is a uniformizer of $M$). Letting $\varpi$ be a uniformizer of $F_v$, we can write $\varpi = u\cdot \varpi_M^e$ using the ramification index $e$ and a unit $u \in \mathcal{O}_M^\times$. From this, by applying Hensel's lemma, we see that $F_v(\varpi^{1/p^\infty}) \subseteq F_{S,w}$. As a result, the maximal tamely ramified pro-$p$ extension of each finite subextension of $F_{S,w}/F_v$ is contained in $F_{S,w}$. Thus, the maximal tamely ramified pro-$p$ extension of $F_{S,w}$ is equal to $F_{S,w}$ itself.

  This implies that $\operatorname{cd}_p(N_{v,S}) = 0$, and therefore $H^j(N_{v,S},\mu_{p^m}^{\otimes n+1}) = 0$ holds for any $n \geq 0$, $m \geq 1$, and $j \geq 1$. Since $H^{j-1}(N_{v,S},\mu_{p^m}^{\otimes n+1})$ is finite, for $j \geq 1$ we have
  \[
    H^j_{\rm cts}(N_{v,S},\mathbb{Z}_p(n+1)) \simeq \varprojlim_m H^j(N_{v,S},\mu_{p^m}^{\otimes n+1}) = 0.
  \]
  By the inflation-restriction exact sequence for continuous cohomology groups analogous to \cite[Proposition (1.6.7)]{NSW}, we conclude that the inflation map
  \[
    H^2_{\rm cts}(D_{v,S},\mathbb{Z}_p(n+1)) \to H^2_{\rm cts}(D_v,\mathbb{Z}_p(n+1))
  \]
  is an isomorphism.
\end{proof}

        \begin{cor}
        \label{segfunwntomnto4t}
  Let $p$ be an odd prime, and let $S$ be a finite set of primes of $F$ satisfying $S \supseteq P_\infty \cup P_p$ (where $P_\infty$ denotes the set of all infinite primes of $F$, and $P_p$ denotes the set of all primes of $F$ lying over $p$). Then, for $n > 0$, there is an isomorphism of abelian groups
  \[
    K_{2n}(\mathcal{O}_F)[p^\infty] \simeq \operatorname{ker}\left[H^2_{\rm cts}(G_{F,S}, \mathbb{Z}_p(n+1)) \xrightarrow{\bigoplus \mathrm{res}_v} \bigoplus_{v \in S - (P_\infty \cup P_p)}H^2_{\rm cts}(D_{v,S}, \mathbb{Z}_p(n+1))\right].
  \]
  If $F$ is totally imaginary, the isomorphism also holds for $p=2$.
\end{cor}

\begin{proof}
    The image of the decomposition group $D_v$ under the natural quotient map $\pi: G_F \twoheadrightarrow G_{F,S}$ is $D_{v,S}$. Therefore, the cohomological map $\mathscr{R}es_v: H_{\rm cts}^i(G_{F,S},\mathbb{Z}_p(n+1)) \to H_{\rm cts}^i(D_v,\mathbb{Z}_p(n+1))$ can be decomposed as
    \[ H^2_{\rm cts}(G_{F,S},\mathbb{Z}_p(n+1)) \xrightarrow{\rm res} H^2_{\rm cts}(D_{v,S},\mathbb{Z}_p(n+1)) \xrightarrow{\rm inf} H^2_{\rm cts}(D_v,\mathbb{Z}_p(n+1)).\]
    Since the inflation map is an isomorphism by Lemma \ref{drghiown4ogtmeo4mo}, the corollary follows.
\end{proof}

  In Proposition \ref{eni4vmo34vo}, we saw the $K$-groups in terms of the cohomology of $G_{F,S}$. However, as mentioned at the beginning of this section, it is difficult to compare the structures of these groups in this form. Therefore, in order to rewrite them in terms of the cohomology of $G_{\mathbb{Q},S}$, we introduce the following proposition.
          
\begin{prop}
  \label{setjfiwnmtgongto4etmnoe4}
  Let $p$ be an odd prime, and let $S \subseteq P_F$ be a finite set of primes of $F$ satisfying $S \supseteq P_\infty \cup P_p \cup \operatorname{ram}(F/\mathbb{Q})$. Furthermore, suppose that there exists a finite set of rational primes $W \subseteq P_{\mathbb{Q},\mathrm{f}}$ (which may be empty) satisfying
  \[
    S - (P_\infty \cup P_p) = \bigsqcup_{q \in W} P_q.
  \]
  Then, $G_{F,S}$ is an open subgroup of $G_{\mathbb{Q},S_{\mathbb{Q}}}$. In the following statements and proofs, for simplicity, we write $G_{\mathbb{Q},S_{\mathbb{Q}}}$ as $G_{\mathbb{Q},S}$ and $D_{q,S_{\mathbb{Q}}}$ as $D_{q,S}$. Let $q$ be a rational prime. Then, for $n > 0$, there exists an isomorphism $\psi$ that makes the following diagram commute:
  \[
    \vcenter{
    \xymatrix@C=10pt{
     H^2_{\rm cts}(G_{\mathbb{Q},S}, \operatorname{Coind}_{G_{F,S}}^{G_{\mathbb{Q},S}}\mathbb{Z}_p(n+1)) \ar[d]^-{\mathrm{res}_q} \ar[rr]^-{\mathrm{sh}_{1}}_-{\sim} & & H^2_{\rm cts}(G_{F,S},\mathbb{Z}_p(n+1))  \ar[d]^-{\bigoplus_{v \mid q} \mathrm{res}_v} \\
      H^2_{\rm cts}(D_{q,S}, \operatorname{Res}_{D_{q,S}}^{G_{\mathbb{Q},S}}\operatorname{Coind}_{G_{F,S}}^{G_{\mathbb{Q},S}}\mathbb{Z}_p(n+1)) \ar[rr]_-{\psi} & & \bigoplus_{v\mid q} H^2_{\rm cts}(D_{v,S},\mathbb{Z} _p(n+1))
    }
    }
  \]
  Here, $\mathrm{sh}_1$ is the morphism on cohomology giving the Shapiro isomorphism defined in Section 2.
\end{prop}

            \begin{proof}

                We choose a set of representatives $\{g_i\}_i \subseteq G_{\mathbb{Q}}$ for the double cosets $G_F\backslash G_{\mathbb{Q}}/D_q$. Fix a prime $v_q$ of $F$ lying over $q$ and a prime $\widetilde{v}_q$ of $\mathbb{Q}^{\rm sep}$ lying over $v_q$; then we obtain a bijection
\[
G_F\backslash G_{\mathbb{Q}}/D_q \to \{v \mid q; \ v \ \text{is a prime of} \ F\}; \quad g_i \mapsto g_i\widetilde{v_q}|_F.
\]
Thus, for each prime $v$ lying over $q$, we denote by $g_v$ the representative $g_i$ that satisfies $g_i\widetilde{v_q}|_F = v$. Writing $\overline{g_v}$ for the image of $g_v$ under the natural surjection $\pi: G_{\mathbb{Q}} \to G_{\mathbb{Q},S}$, the set $\{\overline{g_v}\}$ forms a complete system of representatives for the double cosets $G_{F,S}\backslash G_{\mathbb{Q},S}/D_{q,S}$.

By combining the commutative diagram (Lemma \ref{rmgtiowmto3emtor})
\[
  \xymatrix@C=10pt{
    & H^2_{\rm cts}(G_{\mathbb{Q},S}, \operatorname{Coind}_{G_{F,S}}^{G_{\mathbb{Q},S}}\mathbb{Z}_p(n+1)) \ar[ld]^-{\mathrm{sh}_{\overline{g_v}}} \ar[rd]^-{\mathrm{sh}_1} & \\
    H^2_{\rm cts}(\overline{g_v}^{-1}G_{F,S}\overline{g_v},{}^{\overline{g_v}^{-1}}\mathbb{Z}_p(n+1)) & & H^2_{\rm cts}(G_{F,S},\mathbb{Z}_p(n+1)) \ar[ll]^-{\mathrm{conj}_{\overline{g_v}}}
  }
\]
with the commutative diagram arising from the functoriality of $\operatorname{conj}_{\overline{g_v}}$ (or the restriction map)
\[
  \xymatrix@C=40pt{
H^2_{\rm cts}(\overline{g_v}^{-1}G_{F,S}\overline{g_v},{}^{\overline{g_v}^{-1}}\mathbb{Z}_p(n+1)) \ar[d]_-{\mathrm{res}_v} & H^2_{\rm cts}(G_{F,S},\mathbb{Z}_p(n+1)) \ar[l]^-{\mathrm{conj}_{\overline{g_v}}}_-{\sim} \ar[d]^-{\mathrm{res}_v} \\
H^2_{\rm cts}(\overline{g_v}^{-1}D_{v,S}\overline{g_v},{}^{\overline{g_v}^{-1}}\mathbb{Z}_p(n+1)) & H^2_{\rm cts}(D_{v,S},\mathbb{Z}_p(n+1)) \ar[l]^-{\mathrm{conj}_{\overline{g_v}}}_-{\sim}
  }
\]
it suffices to show that for each $v \mid q$, there exists a morphism $\psi_v$ that makes the diagram
\[
  \vcenter{
  \xymatrix@C=10pt{
   H^2_{\rm cts}(G_{{\mathbb{Q},S}}, \operatorname{Coind}_{G_{F,S}}^{G_{\mathbb{Q},S}}\mathbb{Z}_p(n+1)) \ar[d]^-{\mathrm{res}_q} \ar[rr]^-{\mathrm{sh}_{\overline{g_v}}}_-{\sim} & & H^2_{\rm cts}(\overline{g_v}^{-1}G_{F,S}\overline{g_v},{}^{\overline{g_v}^{-1}}\mathbb{Z}_p(n+1)) \ar[d]^-{\mathrm{res}_v} \\
    H^2_{\rm cts}(D_{q,S}, \operatorname{Res}_{D_{q,S}}^{G_{\mathbb{Q},S}}\operatorname{Coind}_{G_{F,S}}^{G_{\mathbb{Q},S}}\mathbb{Z}_p(n+1)) \ar[rr]_-{\psi_v} & &  H^2_{\rm cts}(\overline{g_v}^{-1}D_{v,S}\overline{g_v},{}^{\overline{g_v}^{-1}}\mathbb{Z}_p(n+1))
  }
  }
\]
commute, and that $\bigoplus_{v \mid q} \psi_v$ is an isomorphism. First, let $\alpha \in H_{\rm cts}^2(G_{\mathbb{Q},S},\operatorname{Coind}_{G_{F,S}}^{G_{\mathbb{Q},S}}\mathbb{Z}_p(n+1))$ be an arbitrary cohomology class. We denote its image under the restriction map by $\alpha_1$:
\[
  \begin{array}{ccc}
H_{\rm cts}^2(G_{\mathbb{Q},S},\operatorname{Coind}_{G_{F,S}}^{G_{\mathbb{Q},S}}\mathbb{Z}_p(n+1))& \stackrel{\mathrm{res}}{\longrightarrow} &  H_{\rm cts}^2(D_{q,S},\operatorname{Res}_{D_{q,S}}^{G_{\mathbb{Q},S}}\operatorname{Coind}_{G_{F,S}}^{G_{\mathbb{Q},S}}\mathbb{Z}_p(n+1))\\
 \rotatebox{90}{$\in$} & & \rotatebox{90}{$\in$}\\
\alpha &\longmapsto  & \alpha_1
\end{array}
\]
where $\alpha_1$ denotes the cohomology class of the cocycle defined by $\alpha_1(g_1,g_2) := \alpha(g_1,g_2)$ ($g_1,g_2 \in D_{q,S}$). By the Mackey decomposition, we have the isomorphism
\[
  \begin{array}{ccc}
\operatorname{Res}_{D_{q,S}}^{G_{\mathbb{Q},S}}\operatorname{Coind}_{G_{F,S}}^{G_{\mathbb{Q},S}}\mathbb{Z}_p(n+1) & \stackrel{\mathrm{Mackey}}{\longrightarrow} &  \bigoplus_{v \mid q} \operatorname{Coind}_{D_{q,S} \cap \overline{g_v}^{-1}G_{F,S}\overline{g_v}}^{D_{q,S}}{}^{\overline{g_v}^{-1}}\mathbb{Z}_p(n+1)\\
 \rotatebox{90}{$\in$} & & \rotatebox{90}{$\in$}\\
f &\longmapsto  & \{d \mapsto f(\overline{g_v}d)\} \quad (d \in D_{q,S}).
\end{array}
\]
We denote the $v$-part of the image of the cohomology class $\alpha_1$ obtained through this decomposition by $\alpha_{1,v}$:
\[
  H_{\rm cts}^2(D_{q,S},\operatorname{Res}_{D_{q,S}}^{G_{\mathbb{Q},S}}\operatorname{Coind}_{G_{F,S}}^{G_{\mathbb{Q},S}}\mathbb{Z}_p(n+1)) \xrightarrow{\sim} \bigoplus_{v \mid q} H_{\rm cts}^2(D_{q,S},\operatorname{Coind}_{\overline{g_v}^{-1}D_{v,S}\overline{g_v}}^{D_{q,S}}{}^{\overline{g_v}^{-1}}\mathbb{Z}_p(n+1)); \quad \alpha_1 \mapsto (\alpha_{1,v})_v
\]
Here, we have used the fact that $D_{q,S} \cap \overline{g_v}^{-1}G_{F,S}\overline{g_v} = \overline{g_v}^{-1}D_{v,S}\overline{g_v}$. By Shapiro's lemma, we obtain the isomorphism
\[
  \begin{array}{ccc}
H_{\rm cts}^2(D_{q,S},\operatorname{Coind}_{\overline{g_v}^{-1}D_{v,S}\overline{g_v}}^{D_{q,S}}{}^{\overline{g_v}^{-1}}\mathbb{Z}_p(n+1)) & \stackrel{\mathrm{res}}{\longrightarrow} &  H_{\rm cts}^2(\overline{g_v}^{-1}D_{v,S}\overline{g_v}, {}^{\overline{g_v}^{-1}}\mathbb{Z}_p(n+1))\\
 \rotatebox{90}{$\in$} & & \rotatebox{90}{$\in$}\\
\alpha_{1,v} &\longmapsto  & \alpha_{2,v}
\end{array}
\]
where $\alpha_{2,v}$ is the cohomology class of the cocycle defined by
\[
  \alpha_{2,v}(g_1',g_2') := \alpha_{1,v}(g_1',g_2')(\overline{g_v}) \quad (g_1',g_2' \in \overline{g_v}^{-1}D_{v,S}\overline{g_v}).
\]
Therefore, if we define the morphism $\psi_v$ as the composition
\begin{align*}
  \psi_v: H_{\rm cts}^2(D_{q,S},\operatorname{Res}_{D_{q,S}}^{G_{\mathbb{Q},S}}\operatorname{Coind}_{G_{F,S}}^{G_{\mathbb{Q},S}}\mathbb{Z}_p(n+1)) &\overset{\sim}{\longrightarrow} \bigoplus_{v \mid q} H_{\rm cts}^2(D_{q,S},\operatorname{Coind}_{\overline{g_v}^{-1}D_{v,S}\overline{g_v}}^{D_{q,S}}{}^{\overline{g_v}^{-1}}\mathbb{Z}_p(n+1)) \\ 
  &\overset{\mathrm{pr}_v}{\longrightarrow} H_{\rm cts}^2(D_{q,S},\operatorname{Coind}_{\overline{g_v}^{-1}D_{v,S}\overline{g_v}}^{D_{q,S}}{}^{\overline{g_v}^{-1}}\mathbb{Z}_p(n+1)) \\
  &\overset{\sim}{\longrightarrow} H_{\rm cts}^2(\overline{g_v}^{-1}D_{v,S}\overline{g_v}, {}^{\overline{g_v}^{-1}}\mathbb{Z}_p(n+1))
\end{align*}
it is clear that $\bigoplus_{v \mid q} \psi_v$ induces an isomorphism. Furthermore, the cohomology class of the image of $\alpha$ under $\psi_v$ becomes
\begin{equation}
\label{drgnuito2ntio4io4t}
  \alpha_{2,v}(g_1',g_2') = \alpha_{1,v}(g_1',g_2')(\overline{g_v}) = \alpha(g_1',g_2')(\overline{g_v}) \quad (g_1',g_2' \in \overline{g_v}^{-1}D_{v,S}\overline{g_v}).
\end{equation}

        On the other hand, by Shapiro's lemma, $\alpha$ maps to $\alpha_{1,v}'$:
\[
  \begin{array}{ccc}
H^2_{\rm cts}(G_{\mathbb{Q},S},\operatorname{Coind}_{G_{F,S}}^{G_{\mathbb{Q},S}}\mathbb{Z}_p(n+1))& \stackrel{\mathrm{sh}_{\overline{g_v}}}{\longrightarrow} &  H^2_{\rm cts}(\overline{g_v}^{-1}G_{F,S}\overline{g_v},{}^{\overline{g_v}^{-1}}\mathbb{Z}_p(n+1))\\
 \rotatebox{90}{$\in$} & & \rotatebox{90}{$\in$}\\
\alpha &\longmapsto  & \alpha_{1,v}'
\end{array}
\]
where $\alpha_{1,v}'$ is defined by
\[
  \alpha_{1,v}'(g_1',g_2') := \alpha(g_1',g_2')(\overline{g_v}) \quad (g_1',g_2' \in \overline{g_v}^{-1}G_{F,S}\overline{g_v}).
\]
The image of this cohomology class under the restriction map clearly coincides with \eqref{drgnuito2ntio4io4t}. This establishes the commutativity of the diagram.
        
            \end{proof}

    \begin{cor}
  Let $p$ be an odd prime, and let $S \subseteq P_F$ be a finite set of primes of $F$ satisfying $S \supseteq P_\infty \cup P_p \cup \operatorname{ram}(F/\mathbb{Q})$. Furthermore, suppose that there exists a finite set of rational primes $W \subseteq P_{\mathbb{Q},\mathrm{f}}$ (which may be empty) satisfying
  \[
    S - (P_\infty \cup P_p) = \bigsqcup_{q \in W} P_q.
  \]
  Then, for $n > 0$, there is an isomorphism of abelian groups
  \begin{align*}
    &K_{2n}(\mathcal{O}_F)[p^\infty] \\
    &\qquad\simeq \operatorname{ker}\left[ H^2_{\rm cts}(G_{\mathbb{Q},S},\operatorname{Coind}_{G_{F,S}}^{G_{\mathbb{Q},S}}\mathbb{Z}_p(n+1)) \xrightarrow{\bigoplus \operatorname{res}_q} \bigoplus_{q \in W} H^2_{\rm cts}(D_{q,S},\operatorname{Res}_{D_{q,S}}^{G_{\mathbb{Q},S}}\operatorname{Coind}_{G_{F,S}}^{G_{\mathbb{Q},S}}\mathbb{Z}_p(n+1))\right].
  \end{align*}
  If $F$ is totally imaginary, the isomorphism also holds for $p=2$.
\end{cor}

 \begin{proof}
     The claim follows immediately from Corollary \ref{segfunwntomnto4t} and Proposition \ref{setjfiwnmtgongto4etmnoe4}.
 \end{proof}

    \section{Main Results}

        In this section, we discuss necessary conditions and sufficient conditions for a pair of number fields $k_1, k_2$ to be $K$-equivalent or almost $K$-equivalent. First, we observe that arithmetical equivalence is a sufficient condition for two number fields to be almost $K$-equivalent. Furthermore, we show that if either $k_1$ or $k_2$ is totally real, the converse holds; that is, arithmetical equivalence is a necessary and sufficient condition for them to be almost $K$-equivalent. We also show that if both $k_1$ and $k_2$ are totally imaginary, then local integral equivalence is a sufficient condition for $K$-equivalence.
    
\begin{defi}
  Two number fields $k_1, k_2$ are said to be \textit{$K$-equivalent} if for all non-negative integers $n$, there is an isomorphism
  \[
    K_n(\mathcal{O}_{k_1}) \simeq K_n(\mathcal{O}_{k_2}).
  \]
  Moreover, if there exists a finite set of rational primes $S \subseteq P_{\mathbb{Q},\mathrm{f}}$ such that for any $p \notin S$ and any $n \geq 0$ there is an isomorphism
  \[
    K_n(\mathcal{O}_{k_1})[p^\infty] \simeq K_n(\mathcal{O}_{k_2})[p^\infty],
  \]
  then $k_1, k_2$ are said to be \textit{almost $K$-equivalent}.
\end{defi}

\begin{them}
  \label{hguhtuoiwntoir3}
  Let $k_1, k_2$ be number fields. Let $L$ be a finite Galois extension of $\mathbb{Q}$ containing both $k_1$ and $k_2$, and set $G := \operatorname{Gal}(L/\mathbb{Q})$, $H_1 := \operatorname{Gal}(L/k_1)$, and $H_2 := \operatorname{Gal}(L/k_2)$. If $p$ is an odd prime such that there is an isomorphism of $\mathbb{Z}[G]$-modules $\operatorname{Ind}_{H_1}^G \mathbb{Z}_p \simeq \operatorname{Ind}_{H_2}^G\mathbb{Z}_p$, then for any $n \geq 0$ there is an isomorphism of abelian groups
  \[
    K_{2n}(\mathcal{O}_{k_1})[p^\infty] \simeq K_{2n}(\mathcal{O}_{k_2})[p^\infty]
  \]
  (where $\operatorname{Ind}_{H_i}^G \mathbb{Z}_p := \mathbb{Z}[G] \otimes_{\mathbb{Z}[H_i]} \mathbb{Z}_p \simeq \mathbb{Z}_p[G/H_i]$). In particular, if either $k_1$ or $k_2$ is totally imaginary, the same holds for $p = 2$.
\end{them}

\begin{proof}
  For each $i=1,2$, we choose a finite set $S_i$ of primes of $k_i$ such that $S_i \supseteq P_{k_i,\infty} \cup P_{k_i,p} \cup \operatorname{ram}(L/\mathbb{Q})_{k_i}$ and, furthermore, there exists a finite set of rational primes $W \subseteq P_{\mathbb{Q},\mathrm{f}}$ (which may be empty) independent of $i$ satisfying
  \[
    S_i - (P_{k_i,\infty} \cup P_{k_i,p}) = \bigsqcup_{q \in W} P_{k_i,q}
  \]
  (such a choice of $S_i$ always exists). Then we have $G_{\mathbb{Q},(S_1)_{\mathbb{Q}}} = G_{\mathbb{Q},(S_2)_{\mathbb{Q}}}$, which we write simply as $G_{\mathbb{Q},S}$. Clearly, $G_{L,S_i(L)}$ and $G_{k_i,S_i}$ are open subgroups of $G_{\mathbb{Q},S}$. For each $i=1,2$, Lemma \ref{rejtgiownmgtiom} yields an isomorphism of $\mathbb{Z}_p[[G_{\mathbb{Q},S}]]$-modules
  \[
    \operatorname{Coind}_{G_{k_i,S_i}}^{G_{\mathbb{Q},S}} (\mathbb{Z}_p) \simeq \operatorname{Inf}_G^{G_{\mathbb{Q},S}}\left(\operatorname{Coind}_{H_i}^{G} (\mathbb{Z}_p)\right).
  \]
  Since the coinduced module and the induced module are isomorphic for a finite group, we have $\operatorname{Coind}_{H_i}^{G} (\mathbb{Z}_p) \simeq \operatorname{Ind}_{H_i}^G (\mathbb{Z}_p)$. Thus, from the assumption of the theorem, we obtain the isomorphism $\operatorname{Coind}_{G_{k_1,S_1}}^{G_{\mathbb{Q},S}} (\mathbb{Z}_p) \simeq \operatorname{Coind}_{G_{k_2,S_2}}^{G_{\mathbb{Q},S}} (\mathbb{Z}_p)$.
      
  Decomposing the action of the original representation by the projection formula (Lemma \ref{gtjwi4tjmo3tji4ot}), we obtain an isomorphism of $\mathbb{Z}_p[[G_{\mathbb{Q},S}]]$-modules
  \[
    \operatorname{Coind}_{G_{k_i,S_i}}^{G_{\mathbb{Q},S}} (\mathbb{Z}_p(n+1)) \simeq \mathbb{Z}_p(n+1) \widehat{\otimes}_{\mathbb{Z}_p} \operatorname{Coind}_{G_{k_i,S_i}}^{G_{\mathbb{Q},S}} (\mathbb{Z}_p).
  \]
  Since the action on $\mathbb{Z}_p(n+1)$ does not depend on $k_i$, the isomorphism
  \[
    \operatorname{Coind}_{G_{k_1,S_1}}^{G_{\mathbb{Q},S}} (\mathbb{Z}_p(n+1)) \simeq \operatorname{Coind}_{G_{k_2,S_2}}^{G_{\mathbb{Q},S}} (\mathbb{Z}_p(n+1))
  \]
  follows. Choosing an isomorphism $\mathcal{F}: \operatorname{Coind}_{G_{k_1,S_1}}^{G_{\mathbb{Q},S}} (\mathbb{Z}_p(n+1)) \overset{\sim}{\to}\operatorname{Coind}_{G_{k_2,S_2}}^{G_{\mathbb{Q},S}} (\mathbb{Z}_p(n+1))$, the functoriality of the restriction map yields an isomorphism $\mathcal{F}_*'$ that makes the diagram
  \[
    \xymatrix@C=50pt{
      H^2_{\rm cts}(G_{\mathbb{Q},S},\operatorname{Coind}_{G_{k_1,S_1}}^{G_{\mathbb{Q},S}}\mathbb{Z}_p(n+1)) \ar[r]^-{\bigoplus_{q \in W} \mathrm{res}_q} \ar[d]^-{\mathcal{F}_*} & \bigoplus_{q \in W} H^2_{\rm cts}(D_{q,S},\operatorname{Res}_{D_{q,S}}^{G_{\mathbb{Q},S}}\operatorname{Coind}_{G_{k_1,S_1}}^{G_{\mathbb{Q},S}}\mathbb{Z}_p(n+1)) \ar[d]^-{\mathcal{F}_*'} \\
      H^2_{\rm cts}(G_{\mathbb{Q},S},\operatorname{Coind}_{G_{k_2,S_2}}^{G_{\mathbb{Q},S}}\mathbb{Z}_p(n+1)) \ar[r]^-{\bigoplus_{q \in W} \mathrm{res}_q} & \bigoplus_{q \in W} H^2_{\rm cts}(D_{q,S},\operatorname{Res}_{D_{q,S}}^{G_{\mathbb{Q},S}}\operatorname{Coind}_{G_{k_2,S_2}}^{G_{\mathbb{Q},S}}\mathbb{Z}_p(n+1))
    }
  \]
  commute (where we write $D_{q,(S_1)_{\mathbb{Q}}} = D_{q,(S_2)_{\mathbb{Q}}}$ simply as $D_{q,S}$). Therefore, the kernels of the two horizontal maps are isomorphic via $\mathcal{F}_*$, which yields the desired isomorphism regarding the $p$-primary part of the $K$-groups for $n > 0$.

        When $n=0$, the $K$-group for a number field $F$ satisfies the isomorphism $K_0(\mathcal{O}_F) \simeq \mathbb{Z} \oplus \operatorname{Cl}(F)$. Therefore, the assertion follows from the next theorem:
\begin{themC}[Perlis]
  If $p$ is a prime such that there is an isomorphism of $\mathbb{Z}[G]$-modules $\operatorname{Ind}_{H_1}^G \mathbb{Z}_p \simeq \operatorname{Ind}_{H_2}^G\mathbb{Z}_p$, then we obtain an isomorphism of abelian groups
  \[
    \operatorname{Cl}(k_1)[p^\infty] \simeq \operatorname{Cl}(k_2)[p^\infty]
  \]
  \cite[Theorem 3]{Perlis2}.
\end{themC}

If either $k_1$ or $k_2$ is totally imaginary, then both are totally imaginary. This is because $\operatorname{Ind}_{H_1}^G\mathbb{Z}_p \simeq \operatorname{Ind}_{H_2}^G\mathbb{Z}_p$ implies $\operatorname{Ind}_{H_1}^G\mathbb{Q} \simeq \operatorname{Ind}_{H_2}^G\mathbb{Q}$, which means that $k_1$ and $k_2$ are arithmetically equivalent. Arithmetically equivalent number fields have the same number of real primes and the same number of complex primes, respectively\cite[Theorem 1]{Perlis}.
    \end{proof}

    \begin{rem}
  In the paper \cite{Perlis2}, Perlis proves Theorem C for a pair of arithmetically equivalent number fields $k_1, k_2$ using their common Galois closure $L$. However, in the same manner as the discussion immediately preceding Corollary \ref{nr4githto4wtio4at}, a similar result can be obtained even if $L$ is replaced by any finite Galois extension of $\mathbb{Q}$ containing $k_1$ and $k_2$.
\end{rem}

Let $k_1, k_2$ be a pair of arithmetically equivalent fields; then they possess a common Galois closure $L$. Letting their respective Galois groups be $G := \operatorname{Gal}(L/\mathbb{Q})$, $H_1 := \operatorname{Gal}(L/k_1)$, and $H_2 := \operatorname{Gal}(L/k_2)$, there is an isomorphism of $\mathbb{Q}[G]$-modules $\operatorname{Ind}_{H_1}^G\mathbb{Q} \simeq \operatorname{Ind}_{H_2}^G\mathbb{Q}$, which implies that $\operatorname{Ind}_{H_1}^G \mathbb{Z}_p \simeq \operatorname{Ind}_{H_2}^G\mathbb{Z}_p$ for almost all primes $p$. We then define the integer $\nu_{k_1,k_2}$ to be the smallest positive integer such that
\[
  p \nmid \nu_{k_1,k_2} \iff \operatorname{Ind}_{H_1}^G \mathbb{Z}_p \simeq \operatorname{Ind}_{H_2}^G\mathbb{Z}_p.
\]
Concerning this integer, it is known that $\nu_{k_1,k_2} \mid [L:k_1]$\cite[Theorem 1]{Perlis2}.

\begin{rem}
    The $\nu_{k_1,k_2}$ defined by Perlis\cite{Perlis2} differs slightly from our definition. In his definition, the primes may occur with multiplicity (i.e., as prime powers). However, the sets of prime factors coincide exactly, so the conclusion is unaffected by which definition one adopts. Here, for simplicity, we adopt the definition given above.
\end{rem}

\begin{cor}
  \label{wengrtuan3oin}
  If $k_1, k_2$ are arithmetically equivalent number fields, then all algebraic $K$-groups of odd degree are isomorphic:
  \[
    K_{2n+1}(\mathcal{O}_{k_1}) \simeq K_{2n+1}(\mathcal{O}_{k_2}).
  \]
  Furthermore, for any prime $p$ not dividing $2\nu_{k_1,k_2}$, there is an isomorphism of $K$-groups
  \[
    K_{2n}(\mathcal{O}_{k_1})[p^\infty] \simeq K_{2n}(\mathcal{O}_{k_2})[p^\infty].
  \]
  Consequently, $k_1$ and $k_2$ are almost $K$-equivalent. In particular, if either $k_1$ or $k_2$ is totally imaginary, the isomorphism $K_{2n}(\mathcal{O}_{k_1})[p^\infty] \simeq K_{2n}(\mathcal{O}_{k_2})[p^\infty]$ holds for any prime $p$ not dividing $\nu_{k_1,k_2}$.
\end{cor}

\begin{proof}
  Regarding the isomorphism of $K$-groups of odd degree, in the case $n = 0$, we have $K_1(\mathcal{O}_{k_i}) \simeq \mathcal{O}_{k_i}^\times$. Since $\mathcal{O}_{k_1}^\times \simeq \mathcal{O}_{k_2}^\times$ holds for arithmetically equivalent number fields $k_1, k_2$, we have $K_1(\mathcal{O}_{k_1}) \simeq K_1(\mathcal{O}_{k_2})$. In the case $n > 0$, we have $K_{2n+1}(k_1) \simeq K_{2n+1}(k_2)$ by \cite[Proposition 1.7]{Phagan}. In general, the isomorphism $K_{2n+1}(R) \simeq K_{2n+1}(\operatorname{Frac}(R))$ holds for the ring of integers $R$ of a global field\cite[Theorem 7]{Weibel3}, so we obtain the isomorphism $K_{2n+1}(\mathcal{O}_{k_1}) \simeq K_{2n+1}(\mathcal{O}_{k_2})$. For $K$-groups of even degree, the isomorphism of the $p$-primary parts follows from Theorem \ref{hguhtuoiwntoir3}.
\end{proof}
        
        \begin{rem}
  Komatsu's result\cite{Komatsu} was as follows: for a prime $p$ not dividing $[L:k_1]$, the $p$-primary parts of the higher algebraic $K$-groups of the rings of integers of $k_1$ and $k_2$ are isomorphic. 
  
  Since $\nu_{k_1,k_2} \mid [L:k_1]$, Corollary \ref{wengrtuan3oin} refines Komatsu's result except for the case $p=2$, and strictly strengthens Komatsu's result when $k_1$ and $k_2$ are totally imaginary (in general, the prime factors of $\nu_{k_1,k_2}$ and those of $[L:k_1]$ can differ; that is, there exist examples of primes that divide $[L:k_1]$ but do not divide $\nu_{k_1,k_2}$. Such an example can be found in \cite[Example 2]{Perlis2}). 

  Regarding Theorem \ref{hguhtuoiwntoir3}, a similar fact is expected to hold for $p=2$ even if $k_1$ and $k_2$ are not totally imaginary. This is a natural conjecture since Komatsu's result induces an isomorphism of the $2$-primary parts for any pair of number fields provided that $2 \nmid [L:k_1]$. For the $2$-primary part of the $K$-groups of the ring of integers of a general number field, Theorem A does not hold as it is, and it is necessary to take into account the contribution of the real primes in \'{e}tale cohomology\cite[Theorem 93]{Weibel3}. However, a similar proof technique might be applicable by investigating the structure of the cohomology groups in more detail.
\end{rem}

\begin{cor}
  \label{seniftw3notmne4to4jet}
  Let $k_1, k_2$ be totally real algebraic number fields. Then, the following conditions are all equivalent:
  \begin{enumerate}[(1)]
    \item $k_1$ and $k_2$ are arithmetically equivalent,
    \item there exists a positive integer $\nu > 0$ such that for any prime $p$ coprime to $\nu$,
    \[
      \sharp K_{n}(\mathcal{O}_{k_1})[p^\infty] = \sharp K_n(\mathcal{O}_{k_2})[p^\infty]
    \]
    holds,
    \item $k_1$ and $k_2$ are almost $K$-equivalent.
  \end{enumerate}
\end{cor}

\begin{proof}
  The implication (3) $\implies$ (2) is clear. Moreover, (2) $\implies$ (1) follows from the following two theorems:
      
  \begin{themD}
    For a totally real algebraic number field $F$, the special values of the Dedekind zeta function of $F$ at negative integers can be written as
    \[
      \zeta_F(1-2k) = (-1)^{kr_1}\frac{\sharp K_{4k-2}(\mathcal{O}_F)_{\rm tors}}{\sharp K_{4k - 1}(\mathcal{O}_F)_{\rm tors}} \quad \textit{up to factors of} \ 2.
    \]
    \cite[Theorem 77]{Weibel3}
  \end{themD}
  \begin{themE}[Oh]
    Let $F$ be a totally real algebraic number field. Given a finite set of rational primes $S \subseteq P_{\mathbb{Q},\mathrm{f}}$, if the $p$-adic valuation $\operatorname{ord}_p(\zeta_F(1-i))$ of the special values of the Dedekind zeta function at negative integers is known for all primes $p \notin S$, then the Dedekind $\zeta$-function of $F$ can be reconstructed as a complex function\cite[Section 3]{Oh}.
  \end{themE}

  Finally, since (1) $\implies$ (3) follows from Corollary \ref{wengrtuan3oin}, the equivalence is established.
\end{proof}

    Let $k_1, k_2$ be number fields and let $L$ be an arbitrary finite Galois extension of $\mathbb{Q}$ containing them. Let their Galois groups be $G := \operatorname{Gal}(L/\mathbb{Q})$, $H_1 := \operatorname{Gal}(L/k_1)$, and $H_2 := \operatorname{Gal}(L/k_2)$. If there is an isomorphism of $\mathbb{Z}[G]$-modules $\operatorname{Ind}_{H_1}^G \mathbb{Z}_p \simeq \operatorname{Ind}_{H_2}^G\mathbb{Z}_p$ for all rational primes $p$, then $k_1, k_2$ are said to be \textit{locally integrally equivalent}\cite[Definition 3.2]{Sutherland} (the word "arbitrary" in the definition can be replaced by "some". Indeed, if the isomorphism of induced modules holds for some $L$, and we take a smaller field $L' \supseteq k_1k_2$ that is Galois over $\mathbb{Q}$, setting $N := \operatorname{Gal}(L/L')$ yields
\[
  \mathbb{Z}_p[(G/N)/(H_1/N)] \simeq \mathbb{Z}_p[G/H_1] \simeq \mathbb{Z}_p[G/H_2] \simeq \mathbb{Z}_p[(G/N)/(H_2/N)].
\]
The same argument applies when taking a larger field $L'$). Clearly, local integral equivalence is a stronger concept than arithmetical equivalence.

\begin{cor}
\label{nr4githto4wtio4at}
  For two number fields $k_1, k_2$ that are locally integrally equivalent, the relations
  \[
    K_{2n+1}(\mathcal{O}_{k_1}) \simeq K_{2n+1}(\mathcal{O}_{k_2}), \quad K_{2n}(\mathcal{O}_{k_1}) \otimes_{\mathbb{Z}} \mathbb{Z}[2^{-1}] \simeq K_{2n}(\mathcal{O}_{k_2}) \otimes_{\mathbb{Z}} \mathbb{Z}[2^{-1}]
  \]
  hold for any non-negative integer $n \geq 0$ (where the zeroth algebraic $K$-groups are isomorphic: $K_0(\mathcal{O}_{k_1}) \simeq K_0(\mathcal{O}_{k_2})$). In particular, if either $k_1$ or $k_2$ is totally imaginary, then $K_{2n}(\mathcal{O}_{k_1}) \simeq K_{2n}(\mathcal{O}_{k_2})$ holds, and these two number fields are $K$-equivalent.
\end{cor}

\begin{proof}
  From Theorem \ref{hguhtuoiwntoir3} and Theorem C, the isomorphism of the $p$-primary parts of the $K$-groups of even degree follows for any odd prime $p$. For a general number field $F$, we have $K_0(\mathcal{O}_F) \simeq \mathbb{Z} \oplus \operatorname{Cl}(F)$, and $K_{2n}(\mathcal{O}_F)$ is a finite abelian group for $n > 0$; thus, we obtain the isomorphism $K_{2n}(\mathcal{O}_{k_1}) \otimes_{\mathbb{Z}} \mathbb{Z}[2^{-1}] \simeq K_{2n}(\mathcal{O}_{k_2}) \otimes_{\mathbb{Z}} \mathbb{Z}[2^{-1}]$ where $2$ is inverted. Moreover, since local integral equivalence implies arithmetical equivalence, the isomorphism of the $K$-groups of odd degree also follows from Corollary \ref{wengrtuan3oin}.
\end{proof}

\begin{rem}
  Even if two number fields are $K$-equivalent, they are not necessarily locally integrally equivalent. Indeed,
  \[
    k_1 := \mathbb{Q}(X^7 - 7X + 3), \quad k_2 := \mathbb{Q}(X^7 + 14X^4 - 42X^2 - 21X + 9)
  \]
  are non-isomorphic and $K$-equivalent, but they are not locally integrally equivalent. The fact that they are $K$-equivalent can be shown as follows: $k_1$ and $k_2$ are arithmetically equivalent\cite{Perlis}, and we can compute $\nu_{k_1,k_2} = 2$\cite[Example 1]{Perlis2}. Therefore, the relations
  \[
    K_{2n+1}(\mathcal{O}_{k_1}) \simeq K_{2n+1}(\mathcal{O}_{k_2}), \quad K_{2n}(\mathcal{O}_{k_1}) \otimes_{\mathbb{Z}} \mathbb{Z}[2^{-1}] \simeq K_{2n}(\mathcal{O}_{k_2}) \otimes_{\mathbb{Z}} \mathbb{Z}[2^{-1}]
  \]
  hold for $n \geq 0$. It remains to verify that $K_{2n}(\mathcal{O}_{k_1})[2^\infty] \simeq K_{2n}(\mathcal{O}_{k_2})[2^\infty]$. For both $k_1$ and $k_2$, there is exactly one prime ideal lying over $2$, and the narrow class groups of $\mathcal{O}_{k_1}$ and $\mathcal{O}_{k_2}$ are trivial (computations were performed using SageMath\cite{Sage}). Thus, $k_1$ and $k_2$ are fields known as $2$-regular fields. In general, if a number field $F$ is a $2$-regular field, the $2$-primary part of its algebraic $K$-groups can be computed as
  \[
    K_{2n}(\mathcal{O}_F)[2^\infty] \simeq \begin{cases}
        (\mathbb{Z}/2)^{r_1} & (n\equiv 1 \pmod{4}) \\
        0 & (\text{otherwise})
    \end{cases}
  \]
  for $n > 0$, where $r_1$ is the number of real primes of $F$\cite[Example 89]{Weibel3}. Since the number of real primes is the same for a pair of arithmetically equivalent fields\cite[Theorem 1]{Perlis}, their $2$-primary parts are all isomorphic as abelian groups (note that since the narrow class groups of $\mathcal{O}_{k_i}$ are trivial, the ideal class groups of $k_1$ and $k_2$ are also trivial, and hence isomorphic).
\end{rem}

\begin{cor}
  \label{rhgtiwo94gniojwtnjowt}
  There exist infinitely many pairs of non-isomorphic and $K$-equivalent number fields.
\end{cor}

\begin{proof}
  In general, if $k_1$ and $k_2$ are locally integrally equivalent, then for any finite Galois extension $M$ of $\mathbb{Q}$, the fields $k_1M$ and $k_2M$ are also locally integrally equivalent. Indeed, letting $L$ be a Galois closure of $\mathbb{Q}$ containing $k_1$ and $k_2$, the field $LM$ is finite Galois over $\mathbb{Q}$ and contains $k_1M$ and $k_2M$. Setting $G := \operatorname{Gal}(LM/\mathbb{Q})$, $H_1 := \operatorname{Gal}(LM/k_1)$, $H_2 := \operatorname{Gal}(LM/k_2)$, and $N := \operatorname{Gal}(LM/M)$, the subgroups corresponding to $k_iM$ are given by $H_i \cap N$ for $i=1,2$.
  
  By assumption, for any prime $p$, the isomorphism $\mathbb{Z}_p[G/H_1] \simeq \mathbb{Z}_p[G/H_2]$ holds as $\mathbb{Z}_p[G]$-modules (employing the argument immediately preceding Corollary \ref{nr4githto4wtio4at}). Tensoring this with $\mathbb{Z}_p[G/N]$ over $\mathbb{Z}_p$ yields the isomorphism
  \begin{align*}
    \mathbb{Z}_p[G/H_1 \times G/N] &\simeq \mathbb{Z}_p[G/H_1] \otimes_{\mathbb{Z}_p} \mathbb{Z}_p[G/N] \\
    &\simeq \mathbb{Z}_p[G/H_2] \otimes_{\mathbb{Z}_p} \mathbb{Z}_p[G/N] \\
    &\simeq \mathbb{Z}_p[G/H_2\times G/N]. 
  \end{align*}
  Since $N$ is a normal subgroup of $G$, every $G$-orbit of $G/H_i \times G/N$ is isomorphic to $G/(H_i \cap N)$. Let $c_i$ denote the number of orbits. then we obtain the direct sum decomposition as modules:
  \[
    \mathbb{Z}_p[G/H_i \times G/N] \simeq (\mathbb{Z}_p[G/(H_i \cap N)])^{c_i}
  \]
  for $i=1,2$, which implies the isomorphism $(\mathbb{Z}_p[G/(H_1 \cap N)])^{c_1} \simeq (\mathbb{Z}_p[G/(H_2 \cap N)])^{c_2}$. These numbers $c_i$ can be computed as
  \[
    c_i = \sharp (G \backslash (G/H_i \times G/N)) = \operatorname{rank}_{\mathbb{Z}_p}(\mathbb{Z}_p[G/H_i \times G/N]_G)
  \]
  (where for a left $G$-module $M$, we define $M_G := M/\langle g\cdot m - m\ | \ g\in G, m \in M\rangle$), from which $c_1 = c_2$ follows. Therefore, the Krull-Schmidt theorem yields $\mathbb{Z}_p[G/(H_1 \cap N)] \simeq \mathbb{Z}_p[G/(H_2 \cap N)]$. Consequently, $k_1M$ and $k_2M$ are locally integrally equivalent.

  There exist infinitely many pairs of non-isomorphic, locally integrally equivalent number fields. In \cite[Theorem 3.9, Remark 3.11]{Sutherland}, it is shown that for a prime $p \equiv \pm 29 \pmod{120}$ and an extension $L/\mathbb{Q}$ with Galois group $G := \mathrm{PSL}_2(\mathbb{F}_p)$ (which always exists), there exist non-conjugate subgroups $H_1, H_2$ of $G$ that are isomorphic to the alternating group $A_5$ such that their fixed fields induce a solvable equivalence (and solvable equivalence implies local integral equivalence). Choose such a prime $p_1$ and consider a pair of non-isomorphic and solvably equivalent number fields $k_{1,1}, k_{1,2}$, and pick a prime $\ell_1$ such that $k_{1,1}\cap \mathbb{Q}(\mu_{\ell_1}) = k_{1,2} \cap \mathbb{Q}(\mu_{\ell_1}) = \mathbb{Q}$. Then, $k_{1,1}\mathbb{Q}(\mu_{\ell_1})$ and $k_{1,2}\mathbb{Q}(\mu_{\ell_1})$ form a pair of non-isomorphic, totally imaginary, locally integrally equivalent fields. The order of $\mathrm{PSL}_2(\mathbb{F}_{p_1})$ is $p_1(p_1^2 - 1)/2$ and the order of the alternating group $A_5$ is $5!/2 = 60$; thus, the degree of the number fields $k_{1,1}\mathbb{Q}(\mu_{\ell_1})$ and $k_{1,2}\mathbb{Q}(\mu_{\ell_1})$ is computed as $p_1(p_1^2-1)(\ell_1 - 1)/120$. By Dirichlet's theorem on arithmetic progressions, there are infinitely many primes $p$ satisfying $p \equiv \pm 29 \pmod{120}$. Thus, by picking a prime $p_2$ greater than $p_1$ that satisfies the congruence, we obtain a new pair of solvably equivalent number fields $k_{2,1}, k_{2,2}$. Furthermore, choosing a prime $\ell_2$ similarly such that $k_{2,1}\cap \mathbb{Q}(\mu_{\ell_2}) = k_{2,2}\cap \mathbb{Q}(\mu_{\ell_2}) = \mathbb{Q}$ and $\ell_2 \geq \ell_1$, the inequality
  \[
    \frac{p_2(p_2^2 - 1)(\ell_2 - 1)}{120} > \frac{p_1(p_1^2 - 1)(\ell_1 - 1)}{120}
  \]
  holds. Hence, $(k_{1,1}\mathbb{Q}(\mu_{\ell_1}), k_{1,2}\mathbb{Q}(\mu_{\ell_1}))$ and $(k_{2,1}\mathbb{Q}(\mu_{\ell_2}), k_{2,2}\mathbb{Q}(\mu_{\ell_2}))$ are distinct pairs of number fields. Repeating this process infinitely many times yields a family of non-isomorphic, totally imaginary, and locally integrally equivalent pairs of number fields $\{(k_{i,1}\mathbb{Q}(\mu_{\ell_i}), k_{i,2}\mathbb{Q}(\mu_{\ell_i}))\}_{i\geq 1}$. Since this constitutes a family of pairs of $K$-equivalent number fields, the corollary follows.
\end{proof}


\begin{thebibliography}{99}
        \bibitem{Perlis} Perlis, Robert, \textit{On the equation $\zeta_{K}(s)= \zeta_{K'}(s)$}, Journal of number theory, \textbf{9} (1977), 342-360.
        \bibitem{Perlis2} Perlis, Robert, \textit{On the class numbers of arithmetically equivalent fields}, Journal of Number Theory, \textbf{10} (1978), 489-509.
        \bibitem{Perlis3} Bart de Smit and Robert Perlis., \textit{Zeta functions do not determine class numbers}, Bulletin of the American Mathematical Society, \textbf{31} (1994), 213–215.
        \bibitem{NSW} Neukirch, Jürgen, Alexander Schmidt, and Kay Wingberg. \textit{Cohomology of number fields}, Vol. 323. Springer Science \& Business Media, 2013.
        \bibitem{Jannsen} U. Jannsen, \textit{Continuous \'{e}tale cohomology,} Math. Ann. \textbf{280} (1988), 207–245, MR 89a:14022
        \bibitem{Phagan} Phagan, Shaver, \textit{Corresponding Abelian Extensions of Integrally Equivalent Number Fields}, Journal of Number Theory \textbf{280} (2025), 88-112.
        \bibitem{Milne} Milne, James S., \textit{\'{e}tale cohomology (PMS-33)}, No. 33, Princeton university press, 1980.
        \bibitem{Milne2} Milne, James S., \textit{Arithmetic duality theorems}, Vol. 4, Charleston, SC: BookSurge, LLC, 2006.
        \bibitem{Oh} Oh, Jangheon, \textit{On zeta functions and Iwasawa modules}, Transactions of the American Mathematical Society \textbf{350} (1998), 3639-3655.
        \bibitem{Sutherland} Sutherland, Andrew V., \textit{Stronger arithmetic equivalence}, preprint arXiv:2104.01956 (2021).
        \bibitem{Komatsu} Komatsu, K., \textit{Eine Bemerkung über Dedekindsche Zetafunktionen und $K$-Gruppe}, Archiv der Mathematik \textbf{54} (1990), 164-165.
        \bibitem{Weibel3} Charles Weibel, \textit{Algebraic $K$-Theory of Rings of Integers in Local and Global Fields,} Springer Books, in: Eric M. Friedlander \& Daniel R. Grayson (ed.), Handbook of $K$-Theory, chapter 0, 139-190, Springer, 2005.
        \bibitem{Sage} SageMath, the Sage Mathematics Software System (Version 10.7). The Sage Developers, 2025, https://www.sagemath.org.
    \end{thebibliography}
\end{document}